\documentclass[10pt]{amsart}

\usepackage{latexsym}
\usepackage{amssymb}
\usepackage{color}
\usepackage[frame,cmtip,arrow,matrix,line,graph]{xy}

\usepackage{amsmath}
\usepackage{amsthm}
\usepackage{amsfonts}
\usepackage{amscd}

\usepackage{epsfig}

\usepackage{caption2}
\usepackage{mathrsfs}

\input xy
\xyoption{all}

\usepackage{lpic,enumerate}
\usepackage{color,amsmath}

\newtheorem{thm}{Theorem}[section]
\newtheorem{lem}[thm]{Lemma}
\newtheorem{prop}[thm]{Proposition}

\newtheorem{cor}[thm]{Corollary}
\newtheorem{Def}[thm]{Definition}

\newcommand{\vs}{\vspace{4mm}}
\newcommand{\no}{\noindent}

\newcommand{\A}{\mathcal{A}}

\newcommand{\BB}{\mathcal{B}}

\newcommand{\CC}{\mathbb{C}}

\newcommand{\tC}{\tilde{C}}

\newcommand{\F}{\mathcal{F}}

\newcommand{\M}{\mathcal{M}}

\newcommand{\OO}{\mathcal{O}}

\newcommand{\QQ}{\mathbb{Q}}

\newcommand{\Z}{\mathbb{Z}}

\newcommand{\al}{\alpha}
\newcommand{\be}{\beta}
\newcommand{\ga}{\gamma}
\newcommand{\Ga}{\Gamma}

\newcommand{\De}{\Delta}

\newcommand{\s}{\sigma}

\newcommand{\surj}{\twoheadrightarrow}

\newcommand{\rar}{\longrightarrow}
\newcommand{\inc}{\hookrightarrow}
\newcommand{\sta}{\stackrel}
\newcommand{\minus}{\backslash}

\newcommand{\x}{\times}
\newcommand{\ot}{\otimes}

\newcommand{\lgl}{\langle}
\newcommand{\rgl}{\rangle}

\newcommand{\del}{\partial}

\newcommand{\emp}{\varnothing}

\newcommand{\Dif}{\operatorname{Diff}}

\newcommand{\Conf}{\operatorname{Conf}^{fr}}
\newcommand{\link}{\operatorname{Link}}
\newcommand{\Star}{\operatorname{Star}}

\definecolor{olgreen}{RGB}{25,90,0}

\newcounter{samcounter}

\date{\today}

\begin{document}
\setcounter{page}{1}
%
%

%
%
\title{Homological stability for mapping class groups of surfaces}

%
%
\author{Nathalie Wahl}
\address{University of Copenhagen}
\email{wahl@math.ku.dk}

%
%

\begin{abstract}
We give a complete and detailed proof of Harer's stability theorem for the homology of mapping class groups of
surfaces, with the best stability range presently known. This theorem and its proof have seen several improvements since Harer's
original proof in the mid-80's, and our purpose here is to assemble these many additions. 
\end{abstract}  

\maketitle
\thispagestyle{empty}
%
%

\section{Introduction}

The purpose of this paper is to give a complete and detailed proof of Harer's stability theorem for the homology of the mapping class
groups of surfaces, with the best known bound.
Harer's paper has been improved a number of times over the past 35 years, by various authors and the argument given here 
attempts to give a ``best of'' from these papers.

\smallskip

Let $S_{g,r}$ denote a surface of genus $g$ with $r$ boundary components. The mapping class group of $S_{g,r}$,
$$\Ga_{g,r}:=\pi_0\Dif(S_{g,r}\ \textrm{rel}\ \del),$$
is the group of components of the orientation preserving diffeomorphism group of $S_{g,r}$, where the diffeomorphisms are assumed to be the identity on the
boundary of $S_{g,r}$.  We consider in this paper the homology of these groups. Recall that the homology of a group $G$ (as
a group) equals the homology of its classifying space $BG$ (as a space), and that the moduli space of Riemann surfaces 
$\M_{g,r}$ is a model for the classifying space $B\Ga_{g,r}$ when $r>0$, and a rational model when $r=0$, i.e.~ 
$H_*(\Ga_{g,r},\Z)\cong H_*(\M_{g,r},\Z)$ when $r>0$ and $H_*(\Ga_{g,0},\QQ)\cong H_*(\M_{g},\QQ)$.

Gluing a pair of pants along two or one of its boundary components define inclusions
$\al\colon  S_{g,r+1}\inc S_{g+1,r}$ and $\beta\colon S_{g,r}\inc S_{g,r+1}$, which 
induce maps on the mapping class groups
$$\al_g\colon  \Ga(S_{g,r+1})\to \Ga(S_{g+1,r}) \ \ \textrm{and}\ \  \beta_g\colon \Ga(S_{g,r})\to \Ga(S_{g,r+1})$$
by extending diffeomorphisms to be the identity on the added pairs of pants. 

Note that $\be_g$ is injective, with left inverse the map 
$$\delta_g\colon  \Ga(S_{g,r+1})\to \Ga(S_{g,r})$$
induced by gluing a disc on one of the newly created boundary components. 
The main theorem proved in this paper is the following:

\begin{thm}\label{main}
Let $g\ge 0$ and $r\ge 1$. The map $$H_*(\al_g)\colon H_*(\Ga(S_{g,r+1}),\Z)\to H_*(\Ga(S_{g+1,r}),\Z)$$ 
is surjective for $*\le \frac{2}{3}g+\frac{1}{3}$ and an isomorphism for 
$*\le \frac{2}{3}g-\frac{2}{3}$. The map 
$$H_*(\beta_g)\colon H_*(\Ga(S_{g,r}),\Z)\to H_*(\Ga(S_{g,r+1}),\Z)$$ 
is always injective and is an isomorphism $*\le\frac{2}{3}g$. 
\end{thm}

Considering $\delta_g\colon \Ga(S_{g,1})\to \Ga(S_{g,0})$ now induced by gluing a disc to the only boundary component 
of $S_{g,1}$, we also get a stability result for closed surfaces:

\begin{thm}\label{mainclosed}
The map $H_*(\delta_g)\colon H_*(\Ga_{g,1},\Z)\to H_*(\Ga_{g,0},\Z)$ is
surjective for $*\le \frac{2}{3}g+1$ and an isomorphism for $*\le \frac{2}{3}g$.  
\end{thm}

Theorems \ref{main} and \ref{mainclosed} were first proved by Harer in \cite{Har85} with a stability bound
of the order of $\frac{1}{3}g$. This was improved to $\frac{1}{2}g$ shortly afterwards
by Ivanov \cite{Iva87,Iva89,Iva93}.  The first complete proof of a 
$\frac{2}{3}g$-range is due to Boldsen, though it is 
based on an earlier unpublished preprint of Harer \cite{Bol09,Har93}. 
We stated in the above theorems the bounds obtained by Randal-Williams in \cite{RW09}. These bounds are a slight
improvement of \cite{Bol09}. With our
knowledge of the stable homology (\cite{Miller,Mor87}, see also \cite{MadWei}) and using recent calculations of Morita 
\cite[Thm.~1.1]{Mor03},  
it follows that this last range is best possible for $g=2$ mod 3, and at most one off the best possible bound otherwise. 

\smallskip

Harer's stability theorem for mapping class groups of surfaces was inspired by the analogous pre-existing theorem
for general linear groups which goes back to Quillen and Borel (see \cite{Bor74,vdK}). 
The general line of argument, due to Quillen,   
is to build for each group in the sequence considered a simplicial complex
with an action of the group, so that the stabilizers of simplices are previous groups in the sequence---the spectral
sequence for the action of the group on the simplicial complex decomposes then the homology of the group in terms of the
homology of the stabilizers of the action, making an inductive argument possible. (The mapping class
groups of surfaces being a 2-parameter family, we will need here two simplicial complexes for each pair $(g,n)$.) 
For this argument to work, the simplicial complexes need to be highly connected and showing this 
high connectivity is the hard part of the proof. 

The simplicial complexes we use here are, as in \cite{RW09}, two ordered arc complexes (defined in Section~\ref{ordsect}) 
for surfaces
with boundaries, and a disc complex (defined in Section~\ref{closedsect}) for closed surfaces. 
The connectivity arguments are a mix of arguments from the papers 
\cite{Har85,Hat91,Iva89,RW09,Wahl08}.

\medskip

The stability for mapping class groups of surfaces has been generalized in several
directions. When considering surfaces with punctures, there are two different generalizations. 
Let $\Ga(S_{g,k}^r)$ and $\Ga(S_{g,k}^{(r)})$ denote the mapping class group of a surface of genus $g$, with $k$ boundary
components and $r$ punctures, where the punctures are assumed to be fixed by the mapping classes in the first case, and to be
fixed up to permutations in the second case.   
Theorems~\ref{main} and \ref{mainclosed} still hold if $\Ga(S_{g,k})$ is replaced by $\Ga(S_{g,k}^r)$ or $\Ga(S_{g,k}^{(r)})$, 
that is the maps $\al,\beta$ and $\delta$ defined above 
are isomorphisms in homology in the same range. This can be deduced from the unpunctured case by a spectral sequence argument (see
\cite{Han}), or by introducing punctures into the proof. 
An additional stabilization map can be defined by increasing the 
number of punctures. For $\Ga(S_{g,k}^{(r)})$, this map induces an isomorphism in homology in a range increasing with the number of punctures \cite[Prop.~1.5]{HatWah07}. 
When the surface is a punctured disc, this is Arnold's classical stability theorem for braid groups \cite{Arn70}. 
Furthermore, there are homological stability theorems 
for spin mapping class groups \cite{Bau04,Har90}, for mapping class groups of non-orientable surfaces \cite{Wahl08}, 
and more generally for moduli spaces of surfaces with certain tangential structures \cite{RW09}. 
For higher dimensional manifolds, the homology of the mapping class groups of
3-manifolds stabilizes by connected sum and boundary connected sum with another 3-manifold \cite{HatWah07}, 
and for simply-connected 4-dimensional manifolds, connected sums with $\CC P^2\#\overline{\CC
P^2}$ gives a stabilization \cite{Gia08}.

\smallskip

The present volume contains a survey \cite{Mad10} about the homology of the moduli space of curves, which gives an overview of the
computation of the stable homology of mapping class groups, following the work of Madsen-Weiss and
Galatius-Madsen-Tillmann-Weiss \cite{GMTW,MadWei}. A survey about more general stability phenomena in the topology of moduli spaces
can be found in \cite{Coh10}.

\medskip

\no
{\em Organization of the paper:}  Section~\ref{ordsect} defines the simplicial complexes used in the proof of
Theorem~\ref{main} and gives their main properties, though the proof of high connectivity of the complexes
is postponed until Section~\ref{connectivity}. 
Section~\ref{SSsect} proves Theorem~\ref{main}
via a spectral sequence argument which builds on Section~\ref{ordsect} and \ref{connectivity}. 
Section~\ref{closedsect} takes care of the case of closed surfaces, proving Theorem~\ref{mainclosed}. 
Finally, the appendix recalls some facts about simplicial complexes and piecewise linear topology, needed in particular
for the connectivity arguments in Section~\ref{connectivity}.

\medskip

The author would like to thank Allen Hatcher, Ib Madsen and in particular 
Oscar Randal-Williams for answering  many questions, 
ranging from  ancient to modern mathematics, as well as Richard Hepworth for a
nice picture. The author was supported by the
Danish National Sciences Research Council  (DNSRC) and the European
Research Council (ERC), as well as by the Danish National Research Foundation (DNRF) through the Centre for Symmetry and Deformation.

\section{The ordered arc complex}\label{ordsect}

In this section, we define the simplicial complexes used to prove homological stability. (Basic definitions and properties of
simplicial complexes are given in the appendix.) The complexes admit actions of corresponding 
mapping class groups of surfaces, and we study the properties of these actions. Propositions~\ref{I1},\ref{I2},\ref{I3} and \ref{I4}
give four key ingredients for the proof of homological stability given in the next section. 

\medskip

Consider $S$ an oriented surface with $\del S\neq\emp$. By an {\em arc} in $S$, we always mean an embedded arc intersecting $\del S$
only at its endpoints and doing so transversally. We work with isotopy classes of arcs, where the isotopies are assumed to fix the 
endpoints of the arcs. 

Let $b_0,b_1$ be two distinct points in $\del S$. We will consider in this section 
collections of arcs with disjointly embeddable interiors, and with endpoints the pair 
$\{b_0,b_1\}$. 
Note that the orientation of the surface induces an ordering on such collections at $b_0$ and at $b_1$.

A collection of arcs with disjoint interiors $\{a_0,\dots,a_p\}$ is called {\em non-separating} if
its  complement  $S\minus(a_0\cup \dots\cup a_p)$ is connected.

\begin{Def}\label{OOdef}
Let $\OO(S,b_0,b_1)$ be the simplicial complex with set of vertices the isotopy classes of non-separating arcs with boundary
$\{b_0,b_1\}$.  A  $p$--simplex
of $\OO(S,b_0,b_1)$ is a collection of $p+1$ distinct isotopy classes of arcs $\lgl a_0,\dots,a_p\rgl$ which 
can be represented by a collection of arcs with disjoint interiors which is non-separating and such that
the anticlockwise ordering of $a_0,\dots,a_p$ at $b_0$ agrees with the clockwise ordering at $b_1$. 

Up to isomorphism, there are two such complexes, depending on whether $b_0$
and $b_1$ are on the same or on different boundary components. 
We will denote by $\OO^1(S)$ the complex with a choice of $b_0,b_1$ on the
same boundary component,  and $\OO^2(S)$ the complex with a choice of $b_0,b_1$ on 
two different boundary components. 
\end{Def} 

\begin{figure}[ht]
\begin{lpic}{OO.3(0.4,0.4)}
 \lbl[b]{40,48;$b_1$}
 \lbl[b]{-2,49;$b_0$}
 \lbl[b]{195,41;$b_1$}
 \lbl[b]{142,41;$b_0$}
\end{lpic}
\caption{1-simplex of $\OO^1(S)$ and 2-simplex of $\OO^2(S)$}\label{OO}
\end{figure}

The action of the mapping class group $\Ga(S)=\pi_0\Dif(S,\del S)$ on the surface $S$ induces an action on $\OO(S,b_0,b_1)$. 
The remaining of the section gives four properties of this
action.

\begin{prop}[Ingredient 1]\label{I1} For the complex $\OO^i(S_{g,r})$, $i=1,2$, we have: 

{\rm (1)} $\Ga(S_{g,r})$ acts transitively on $p$--simplices for each $p$.

{\rm (2)} There exists isomorphisms 
$$St_{\OO^1}(\s_p) \mathop{\sta{s_1}{\rar}}_{\cong} \Ga(S_{g-p-1,r+p+1}) \ \ \ \textrm{and}\ \ \   
St_{\OO^2}(\s_p) \mathop{\sta{s_2}{\rar}}_{\cong} \Ga(S_{g-p,r+p-1}),$$ 
where $St_{\OO^i}(\s_p)$ denotes the stabilizer of a $p$--simplex $\s_p$ of $\OO^i(S)$. 
\end{prop}

\begin{proof}
Let $i=1,2$ and 
$\s=\lgl a_0,\dots,a_p\rgl$ be a $p$--simplex of $\OO^i(S)$ represented by arcs $a_0,\dots,a_p$ with disjoint interiors in $S$. We consider 
the surface $S$ ``cut along $\s$'', i.e. the surface $S\minus\s=S\minus
(N_0\cup \dots \cup N_p)$ with $N_j$ a small neighborhood of $a_j$. Its Euler characteristic satisfies 
$\chi(S\minus\s)=\chi(S)+p+1$ as a cellular decomposition of $S\minus \s$ can be obtained from one of $S$ by doubling the arcs $a_j$. 
We can moreover count (and describe) the boundary components of $S\minus \s$: in addition to the $r-i$ components of
$\del S$ disjoint from $b_0,b_1$, $\del(S\minus \s)$ has \\
- (when $i=1$) $p+2$ components labeled $[\del_0^+\!S*a_0],[\bar a_0*a_1],\dots, [\bar a_{p-1}*a_p],[\bar a_p*\del_0^-\!S]$, \\
- (when $i=2$) $p+1$ components labeled $[\del_0S*a_0*\del_1S*\bar a_p],[\bar a_0*a_1],\dots, [\bar a_{p-1}*a_p]$, \\
where $a_i,\bar a_i$ denote the left and right side of the arc, and $\del_0^+\!S, \del_0^-\!S, \del_0S$ and $\del_1S$ are as shown in
Figure~\ref{OO.cut}. 
\begin{figure}[ht]
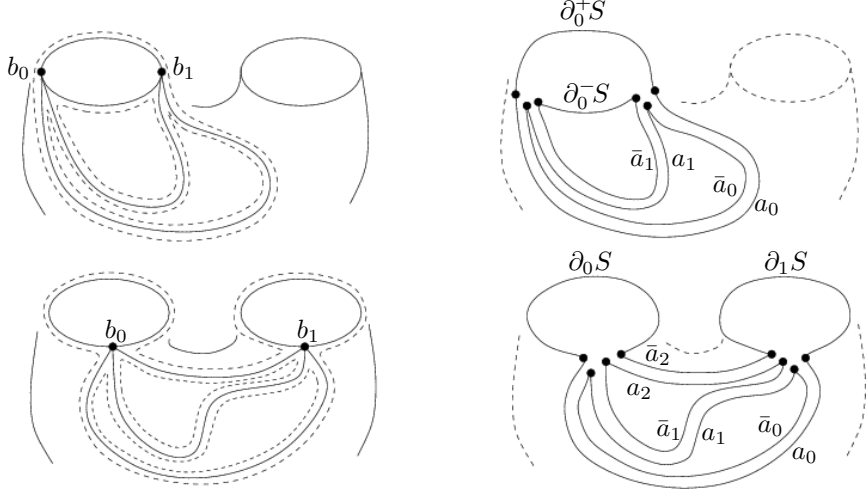

\begin{lpic}{OO.cut.3(0.5,0.5)}
 \lbl[t]{-1,115;$b_0$}
 \lbl[t]{43,115;$b_1$}
 \lbl[t]{25,45;$b_0$}
 \lbl[t]{76,45;$b_1$}
 \lbl[t]{198,77;$a_0$}
 \lbl[t]{187,83;$\bar a_0$}
 \lbl[t]{176,89;$a_1$}
 \lbl[t]{165,90;$\bar a_1$}
 \lbl[t]{149,131;$\del_0^+\!S$}
 \lbl[t]{150,109;$\del^-_0\!S$}
\lbl[t]{151,63;$\del_0S$}
 \lbl[t]{203,63;$\del_1S$}
 \lbl[t]{208,11;$a_0$}
 \lbl[t]{199,19;$\bar a_0$}
 \lbl[t]{184,16;$a_1$}
 \lbl[t]{172,18;$\bar a_1$}
\lbl[t]{164,28;$a_2$}
 \lbl[t]{169,38;$\bar a_2$}
\end{lpic}
\caption{Cutting along the simplices of Figure~\ref{OO}}\label{OO.cut}
\end{figure}

As $S\minus \s$ is connected by assumption, we have $S\minus \s\cong S_{g_\s,r_\s}$ with Euler characteristic
$$\chi(S\minus\s)=\ \ 2-2g_\s-r_\s \ \ =\ \  2-2g-r+p+1\ \ =\chi(S)+p+1$$
By the above when $i=1$, $r_\s=r+p+1$ and thus $g_\s=g-p-1$, and when $i=2$, $r_\s=r+p-1$ and thus $g_\s=g-p$. 

As $g_\s$ and $r_\s$ depend only on $p$, not on the simplex itself, it follows that the complement of any two $p$--simplices $\s,\s'$
are diffeomorphic.  Moreover, 
we can choose a diffeomorphism $S\minus\s\cong S\minus\s'$ which is compatible with the labels of the boundary by arcs
of $\s$ (resp.~$\s'$), 
and hence glues to a
diffeomorphism of $S$ taking $\s$ to $\s'$. Property (1) in the proposition follows.

\medskip

The map $\Ga(S\minus\s)\to \Ga(S)$ which glues $S\minus \s$ back along $\s$ has image a subgroup of $St_{\OO^i}(\s)$. We want to show that this
map defines an isomorphism  $\Ga(S\minus\s)\cong St_{\OO^i}(\s)$. The second part of the proposition will then follow as, 
by the above, $S\minus\s$ is a surface of type $S_{g-p-1,r+p+1}$ 
when $i=1$ and of type $S_{g-p,r+p-1}$ when $i=2$. 

To check surjectivity, 
consider an element $\phi$ of the stabilizer of the above simplex $\s$. Stabilizing $\s$ means that for each $i$ we have an isotopy 
$\phi(a_i)\simeq_{h_i}a_{\theta(i)}$ for some
permutation $\theta$ of $\{0,1,\dots,p\}$. We want to show that $\theta$ is the identity and $\phi$ is isotopic to a map that fixes 
the arcs pointwise.

By the isotopy extension theorem, the isotopy $h_0$ can be extended to an ambient isotopy, i.e. levelwise diffeomorphisms
 $H_0\colon S\x I\to S$  with $H_0(S,0)=id_S$ and $H_0(\phi(a_0),t)=h_0(\phi(a_0),t)$. Composing with the end diffeomorphism gives a map  
$\phi_1=H_0(1,-)\circ \phi$ which satisfies $\phi_1\simeq \phi$ and $\phi_1(a_0)=a_{\theta(0)}$. 

Now suppose that we have constructed $\phi_i\simeq \phi$ with $\phi_i(a_j)=a_{\theta(j)}$ for each $j<i$. Consider 
$h_i'\colon I\x I\to S$, the isotopy taking $\phi_i(a_i)$ to $a_{\theta(i)}$ . Approximate $h_i'$ by a map $\tilde h_i$ 
transverse to $a_{\theta(0)},\dots,a_{\theta(i-1)}$. We will inductively make it disjoint from these arcs.  
Suppose that the image of $\tilde h_i$ is disjoint from $a_{\theta(0)},\dots,a_{\theta(k-1)}$ and consider $\tilde h_i$ as a map to 
$S\minus(a_{\theta(0)}\cup \dots\cup a_{\theta(k-1)})$. We want to make it disjoint from
$a_{\theta(k)}$. 
The inverse image of $a_{\theta(k)}$ is a union of circles and intervals. For each circle
component (or arc component going back to the same endpoint), $\tilde h_i$ restricted to the disc it separates in $I\x I$ defines
an element of $\pi_2(S\minus(a_{\theta(0)}\cup \dots\cup a_{\theta(k-1)}))$ as its boundary is mapped to a subarc of $a_{\theta(k)}$. As this $\pi_2$ 
is trivial, $\tilde h_i$ is homotopic to a map
$g_i$ with no such circle intersections, i.e. such that either $g_i^{-1}(a_{\theta(k)})$ is a (possibly empty) 
union of intervals from $g_i^{-1}(b_0)$ to $g_i^{-1}(b_1)$. 
If the union in non-empty, $g_i$ would restrict to a homotopy 
$a_{\theta(k)}\simeq a_{\theta(i)}$. As homotopic arcs are isotopic by \cite[Thm 3.1]{Epstein}, this would contradict the fact that $a_{\theta(k)}$
and $a_{\theta(i)}$ represent different vertices of a simplex. 
Hence, proceeding inductively, we can replace the isotopy $h_i$ by a homotopy
between $\phi(a_i)$ and $a_{\theta(i)}$ in  $S\minus(a_{\theta(0)}\cup\dots\cup a_{\theta(i-1)})$. Applying again \cite[Thm 3.1]{Epstein}, 
we get an isotopy in the
same surface, which we can replace by an ambient isotopy $H_i$. Considering $H_i$ as an isotopy of $S$ fixed on
$a_{\theta(0)},\dots,a_{\theta(i-1)}$, we define $\phi_{i+1}=H_i(1,-)\circ \phi_i$.  Repeating the
construction, we obtain $\phi_{p+1}\simeq \phi$ satisfying $\phi_{p+1}(a_i)=a_{\theta(i)}$ for each $i$. 

We finally note that $\theta$ must be the identity as $\phi_{p+1}$ fixes $\del S$, and hence must be isotopic to the identity in a
neighborhood of $\del S$. 

Hence any element of the stabilizer of $\s$ is in the image of $\Ga(S\minus \s)\to\Ga(S)$. 
We are left to show injectivity of that map. This can be seen as follows: To a non-separating arc $I$ in $S$ is associated a fibration
 $$\Dif(S \ \textrm{rel}\ \del S\cup I)\to\Dif(S\ \textrm{rel}\ \del S)\to Emb^\del_{ns}(I,S)$$
where $Emb^\del_{ns}(I,S)$ denotes the space of embeddings of a non-separating arc in $S$ with $\del I$ mapping to chosen points $A,B\in\del S$. The
fibration is induced by restricting a diffeomorphism of $S$ to the given arc. 
By \cite[Thm 5]{Gramain}, $\pi_1(Emb^\del(I,S))=0$.  Hence the long exact sequence of homotopy groups of the fibrations 
gives an injection $\pi_0(\Dif(S \ \textrm{rel}\ \del S\cup I))\inc \pi_0(\Dif(S\ \textrm{rel}\ \del S))$, i.e. 
$\Ga(S\minus I)\inc \Ga(S)$. This corresponds to the case where $\s$ is a vertex. 
Repetitive use of this inclusion gives the desired result for any simplex $\s$. 
\end{proof}

Consider the maps 
$\al\colon S_{g,r+1}\to S_{g+1,r}$ and $\beta\colon S_{g,r}\to S_{g,r+1}$
which glue a strip respectively on two and one boundary components (see Figure~\ref{ab}). 
There are induced maps on the mapping class groups
$$\al_g\colon \Ga(S_{g,r+1})\to \Ga(S_{g+1,r}) \ \ \textrm{and}\ \  \beta_g\colon \Ga(S_{g,r})\to \Ga(S_{g,r+1})$$
by extending the
mapping classes to be the identity on the added strips. 
(These maps are isomorphic to the maps $\al_g$ and $\beta_g$ described in the introduction in terms of pairs of pants.)
Moreover, if $b_0,b_1$ are as in the Figure~\ref{ab}, 
we get induced maps 
$$\al\colon \OO^2(S_{g,r+1})\to \OO^1(S_{g+1,r}) \ \ \textrm{and}\ \  \beta\colon\OO^1(S_{g,r})\to \OO^2(S_{g,r+1})$$
equivariant with respect to the group maps $\al_g$ and $\beta_g$.

\begin{figure}[ht]
\begin{lpic}{ab.2(0.45,0.45)}
 \lbl[b]{158,46;$b_1$}
 \lbl[b]{114,46;$b_0$}
 \lbl[b]{71,41;$b_1$}
 \lbl[b]{20,41;$b_0$}
 \lbl[b]{44,53;$\al$} 
 \lbl[b]{135,45;$\be$}
\end{lpic}
\caption{The maps $\al$ and $\be$}\label{ab}
\end{figure}

The second ingredient is a new ingredient in Randal-Williams' proof. It is a symmetry property of the maps $\al$ and
$\be$ which will be crucial in
the spectral sequence argument. It roughly says that, on stabilizers of the action of the mapping class group $\Ga(S)$ on the complexes
$\OO^i(S)$, 
$\al$ induces $\beta$ and $\beta$ induces $\al$.

\begin{prop}[Ingredient 2]\label{I2}
Let $\al,\beta$ be as in Figure~\ref{ab}. Given a $p$--simplex $\s_p$ in
$\OO^2(S)$, we have
$$\xymatrix{\Ga(S_{g,r+1})\ar[d]^{\al_g} & \ar@{_(->}[l] St_{\OO^2}(\s_p) \ar[r]^-{s_2}_-\cong \ar[d] & \Ga(S_{g-p,r+p}) \ar[d]^-{\beta_{g-p}} \\
\Ga(S_{g+1,r}) &  \ar@{_(->}[l] St_{\OO^1}(\al(\s_p)) \ar[r]^-{s_1}_-\cong & \Ga(S_{g-p,r+p+1}) }$$
and given  a $p$--simplex $\s_p$ in $\OO^1(S)$, we have
$$\xymatrix{\Ga(S_{g,r})\ar[d]^{\be_g} & \ar@{_(->}[l] St_{\OO^1}(\s_p) \ar[r]^-{s_1}_-\cong \ar[d] & \Ga(S_{g-p-1,r+p+1}) \ar[d]^-{\al_{g-p-1}} \\
\Ga(S_{g,r+1}) &  \ar@{_(->}[l] St_{\OO^2}(\be(\s_p)) \ar[r]^-{s_2}_-\cong & \Ga(S_{g-p,r+p}) }$$
where $s_1,s_2$ are the isomorphisms of Proposition~\ref{I1}. 
\end{prop}

\begin{proof}
Consider first the map $\al$. On $S$, $\al$ is defined by gluing a strip, one side glued to $\del_0S$ and one to $\del_1S$. In 
$S\minus\s_p$, the two components are part of a single boundary component, the one denoted $[a_0*\del_0S*\bar a_p*\del_1S]$ in the proof of
Proposition~\ref{I1}. Hence the strip in glued to a single boundary in $S\minus\s_p$. As the stabilizer of $\s_p$ is identified with 
the mapping class
group of $S\minus\s_p$, $\al$ induces the map $\beta$ on $St_{\OO_2}(\s_p)$.

For $\beta$, both ends of the strip are glued to $\del_0S$: one end to $\del_0S^+$ and one to $\del_0S^-$ in the notation of
Proposition~\ref{I1} (see also Figure~\ref{OO.cut}). Given a $p$--simplex $\s_p$ of
$\OO_1(S)$, we have that $\del_0S^+$ and $\del_0S^-$ are in two different boundary components of $S\minus\s_p$, namely the components
denoted   $[\del_0^+\!S*a_0]$ and $[\bar a_p*\del_0^-\!S]$ in the proof of Proposition~\ref{I1}. 
Hence the map $\beta$ induces the map $\al$ on
$St_{\OO_1}(\s_p)$.  
\end{proof}

The complexes $\OO^1$ and $\OO^2$ are defined to ``undo'' the maps $\al$ and $\beta$ in the sense that the inclusion of stabilizers of
vertices $St_{\OO_i}(\s_0)\inc \Ga(S)$ are the maps $\al$ and $\beta$ respectively. The third ingredient makes this precise, as part of
a stronger statement.

\begin{prop}[Ingredient 3]\label{I3}
Let $S_\al$ and $S_\be$ denote $S$ union a strip glued via $\al$ and $\beta$ respectively as in Figure~\ref{ab}.
The maps $\al\colon\Ga(S)\to \Ga(S_\al)$ and $\be\colon\Ga(S)\to \Ga(S_\be)$ are injective. Moreover,  
for any vertex $\s_0$ of $\OO^i(S)$, there are curves
$c_\al,c_\be$ (given in Figure~\ref{curves}) in $S_\al$ and $S_\be$
such that conjugation by  Dehn twists $t_{c_\al}$ and $t_{c_\be}$
along these curves fits into commutative diagrams 
$$\xymatrix{St_{\OO^2}(\s_0)\ar@{^(->}[d] \ar@{^(->}[r] & St_{\OO^1}(\al(\s_0))\ar@{^(->}[d]\ar@{-->}[dl]_{t_{c_\al}}
& & St_{\OO^1}(\s_0)\ar@{^(->}[d] \ar@{^(->}[r] & St_{\OO^2}(\be(\s_0))\ar@{^(->}[d]\ar@{-->}[dl]_{t_{c_\be}} \\
\Ga(S) \ar@{^(->}[r]^{\al} & \Ga(S_\al) & & \Ga(S) \ar@{^(->}[r]^{\be} & \Ga(S_\beta)}$$
i.e. there are conjugations  $St_{\OO^1}(\al(\s_0))\sim_{t_{c_\al}}\al(\Ga(S))$ in $\Ga(S_\al)$ relative to  $\al(St_{\OO^2}(\s_0))$, and 
$St_{\OO^2}(\beta(\s_0))\sim_{t_{c_\be}}\be(\Ga(S))$ in $\Ga(S_\be)$ relative to  $\be(St_{\OO^1}(\s_0))$.
\end{prop}

\begin{figure}[ht]
\begin{lpic}{curves.2(0.42,0.42)}
 \lbl[b]{53,69;$\al$}
 \lbl[b]{190,60;$\be$}
 \lbl[b]{20,31;$a_0$} 
 \lbl[b]{155,60;$a_0$}
 \lbl[b]{96,32;\textcolor{olgreen}{$c_\al$}}
 \lbl[b]{221,60;\textcolor{olgreen}{$c_\be$}}
 \lbl[b]{109,0;(a)}
 \lbl[b]{235,0;(b)}
\end{lpic}
\caption{The curves $c_\al$ and $c_\beta$ of Proposition~\ref{I3} for $\s_0=\lgl a_0\rgl$}\label{curves}
\end{figure}

Note that the stabilizer of any two vertices of $\OO^i(S)$ are conjugate in $\Ga(S)$ as $\Ga(S)$ acts transitively on
the vertices of $\OO^i(S)$. So the proposition implies in particular that the stabilizer of a vertex in $\OO^1(S_\al)$ is isomorphic
to $\Ga(S)$, and that so is the stabilizer of a vertex in $\OO^2(S_\be)$.

The existence of the particular Dehn twist used
($t_{c_\al}$ and $t_{c_\be}$) is also used by  
Harer and Boldsen in their proof the $2/3$ stability range.  In
Randal-Williams, the above property is called
$1$--triviality.

\begin{proof}
 Let $\bar S$ denote either $S_\al$ or $S_\be$. 
Suppose $\s_0$ is represented by an arc $a_0$ in $S$, and denote also by $a_0$ the corresponding arc in $\bar S$. 
The first step of the proof in both cases is to exhibit an arc $a_1$ in $\bar S$ 
disjoint from $a_0$ (except at the
endpoints) such that $St_{\OO^1}(a_1)=\al(\Ga(S))$ (resp.~$St_{\OO^2}(a_1)=\be(\Ga(S))$) and  
$St_{\OO^1}(\lgl a_0,a_1\rgl)=\al(St_{\OO^2}(a_0))$ (resp.~$St_{\OO^2}(\lgl a_0,a_1\rgl)=\al(St_{\OO^1}(a_0))$) 
as subgroups of $\Ga(\bar S)$. 
In other words, we will show that both diagrams can be seen
as being of the form 
$$\xymatrix{St(\lgl a_0,a_1\rgl) \ar@{^(->}[r]\ar@{^(->}[d] & St(a_0) \ar@{^(->}[d]\\
St(a_1)\ar@{^(->}[r] & \Ga(\bar S)}$$  
The arc $a_1$ is given in Figure~\ref{a1phi}(a),(c) in both cases. To show that it satisfies the above,  
it is enough to produce a map $\phi\in \Dif(\bar S)$ such that 

(1) $\phi\simeq id$ and $\phi$ is constant on $a_0$,

(2) $\phi$ takes a neighborhood of $a_1\cup \del \bar S$ to a neighborhood of $\del S \cup (\bar S\minus S)$. \\
Indeed, an element $g\in St(a_1)\le \Ga(\bar S)$ can be assumed to fix a neighborhood of $a_1$ and $\del \bar S$. Then its conjugate 
$c_\phi(g)=\phi^{-1} g\phi$ fixes the strip $\bar S\minus S$ and a neighborhood of $\del S$, and hence is in the image of $\Ga(S)$. 
And conversely, $c_{\phi^{-1}}(g)$ identifies elements in the image of $\Ga(S)$ with elements of $St(a_1)$. 
But as $\phi\simeq id$, $g$ and
$c_\phi(g)$ (resp. $c_{\phi^{-1}}(g)$) represent the same element in $\Ga(\bar S)$ and $c_\phi$ is actually the identity. 
As $a_0$ is not affected, $c_\phi$ restricts to an identification of  $St(\lgl a_0,a_1\rgl)$ with the image of
$St(a_0)$ in $\Ga(\bar S)$. 

The map $\phi$ is obtained by thickening the neighborhood of $\del\bar S\cup a_1$ to include the shaded areas of $\bar S\minus S$ shown
in Figure~\ref{a1phi}(b),(d).  This is possible as the shaded areas are discs intersecting $\del\bar S\cup a_1$ in an arc in their boundary.

\begin{figure}[ht]
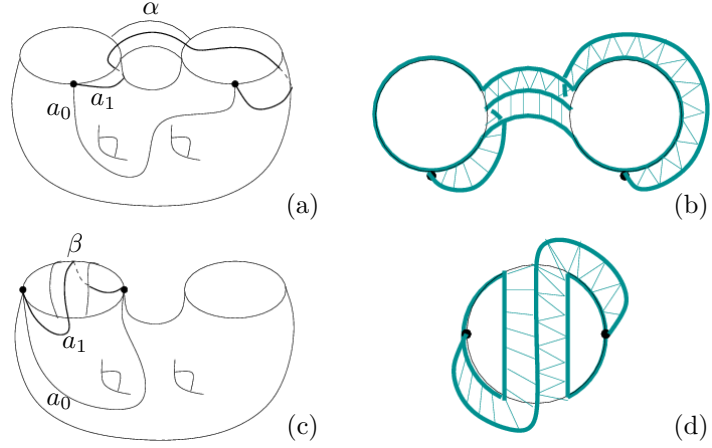

\begin{lpic}{a1phi.2(0.42,0.42)}
 \lbl[b]{44,135;$\al$} 
 \lbl[b]{20,58;$\be$}
 \lbl[b]{15,102;$a_0$} 
 \lbl[b]{29,105;$a_1$}
 \lbl[b]{20,28;$a_1$} 
 \lbl[b]{15,10;$a_0$}
 \lbl[b]{92,70;(a)}
 \lbl[b]{215,70;(b)}
 \lbl[b]{92,0;(c)}
 \lbl[b]{215,0;(d)}
\end{lpic}
\caption{The arc $a_1$ and the retraction $\phi$}\label{a1phi}
\end{figure}

Now in both cases, the arcs $a_0,a_1$ are ordered in the same way at $b_0$ and $b_1$ (that is they do {\em not} form a 1-simplex of
$\OO^i$). It was noted by Harer (\cite{Har93}, see also \cite[Prop. 3.2]{Bol09}) that in this case, the Dehn twist $\psi$ 
along the closed curve $[a_0*a_1]$ takes $a_1$ to $a_0$. Conjugation by $\psi^{-1}$ takes $St(a_0)$ to $St(a_1)$.  
As it lives in a neighborhood of $a_0,a_1$, it commutes with the elements of $St(a_0,a_1)$ and
hence conjugation by $\psi$ is the identity on that subgroup.  (The curve
$[a_0*a_1]$ is the curve drawn in Figure~\ref{curves} in each case.)

The above gives conjugations between the image of $\Ga(S)$ in $\Ga(\bar S)$ and the subgroups $St_{\OO^i}(a_0)$. Injectivity of the
maps $\al$ and $\beta$ then follows from the fact, proved in Proposition~\ref{I1}, that 
 $\Ga(S)$ is isomorphic to $St_{\OO^i}(a_0)$. 
\end{proof}

To be able to use Ingredient 3, we need to study the commutative square
occuring in Proposition~\ref{I3} from the point of view of group homology,
which we do now. 

\smallskip

Given a group $G$, we use the bar construction to compute the homology of
$G$. Hence a $k$-chain $c\in C_k(G)$ is of the form
$c=\sum_iz_i(g^i_1,\dots,g^i_k)$ with $g^i_j\in G$ and $z_i\in \Z$. By a {\em commuting diagram}, we will mean a diagram of
the form 
$$\xymatrix{\star \ar[r]^{a_1}\ar[d]_{b_0}&\star\ar[r]\ar[d]^{b_1} &\dots &
  \star\ar[r]^{a_k} \ar[d]_{b_{k-1}}&\star\ar[d]^{b_k}\\
 \star \ar[r]^{c_1}&\star\ar[r] &\dots &\star\ar[r]^{c_k}&\star}$$
 with $a_i,b_i,c_i\in G$ and $a_ib_i=b_{i-1}c_i$ for each $i$. 
Such a commuting diagram defines a copy of $\Delta^k\times I$ in the
classifying space of $G$ and hence an element of $C_{k+1}(G)$. Explicitly (and
with a choice of orientation), this chain is 
$$(a_1,\dots,a_k,b_k)-
(a_1,\dots,a_{k-1},b_{k-1},c_k)+\dots+(-1)^k(b_0,c_1,\dots,c_k).$$

\begin{lem}\label{relative}
Let $H,G_1,G_2$ be subgroups of a group $G$ 
fitting into a diagram
$$\xymatrix{H\ \ar@{^(->}[r]\ar@{^(->}[d] & \ G_1\ar@{^(->}[d]\ar@{-->}_t[dl]\\
G_2\ \ar@{^(->}[r] & \ G}$$
i.e.~such that $G_1$ and $G_2$ are conjugated by $t\in G$, with a common subgroup $H$
fixed by the conjugation. Then the map
$H_k(G_1,H)\rar H_k(G,G_2)$ induced by the inclusion factors as
$$\xymatrix{H_k(G_1,H) \ar[r]\ar[d]^\del &
  H_k(G,G_2)\\
H_{k-1}(H)\ar[ur]_{(-1)^k(\,\_\,\x t)}}$$
where $(h_1,\dots,h_{k-1}) \x t$ is the class given by the
commuting diagram
$$\xymatrix{\star \ar[r]^{h_1}\ar[d]_{t} & \star\ar[d]^{t}\ar[r] & \dots & \star \ar[d]\ar[r]^{h_{k-1}} &
  \star \ar[d]^{t}\\
\star \ar[r]^{h_1} & \star\ar[r] & \dots & \star \ar[r]^{h_{k-1}} & \star}$$
\end{lem}

\begin{proof}
A class $[c]\in H_k(G_1,H)$ is of the form $c=\sum_iz_i(g_1^i,\dots,g_k^i)$ with each  
$g^i_j\in G_1$ and $z_i\in\Z$, and with boundary $dc$ a chain in $H$. Extending the above
notation, denote by $c\times t$ the $(k+1)$-chain given by the linear combination of commuting
diagrams 
 $$\xymatrix{\star \ar[r]^{g^i_1}\ar[d]_{t} & \star\ar[d]^{t}\ar[r] & \dots & \star \ar[d]\ar[r]^{g^i_k} &
  \star \ar[d]^{t}\\
\star \ar[r]_{t^{-1}g^i_1t} & \star\ar[r] & \dots & \star \ar[r]_{t^{-1}g^i_kt} & \star}$$
One computes that $d(c\times t)=(-1)^{k+1}c+dc\times t+(-1)^kt^{-1}ct$. (In
particular, the map $\_\,\x t:C_{k}(H)\to C_{k+1}(G)$ is a chain map as $t^{-1}ct=c$
for any $c\in C_{k}(H)$.) As $t$ conjugates $G_1$ into $G_2$, we have that 
$[t^{-1}ct]=0$ in $H_*(G,G_2)$, and hence the image of $[c]$ in  $H_*(G,G_2)$ is equal to 
$(-1)^k[dc\times t]$. 
\end{proof}

We will use Proposition~\ref{I3} via the following two corollaries. 

\begin{cor}\label{I3cora}
Let $\s_0$ be a vertex of $\OO^2(S)$ and $\al(\s_0)$ its image in $\OO^1(S_\al)$. 
Then the map induced on relative homology by including the stabilizers 
$$H_*(St_{\OO^1}(\al(\s_0)),St_{\OO^2}(\s_0))\rar
H_*(\Ga(S_\al),\Ga(S))$$
is the zero map. 
\end{cor}

\begin{proof}
Applying Lemma \ref{relative} to Proposition~\ref{I3}, it is enough to show
that the map  $$\_\x t_{c_\al}: H_{*-1}(St_{\OO^2}(\s_0))\rar
H_*(\Ga(S_\al),\Ga(S))$$
is the zero map. Recall that $t_{c_\al}$ is a Dehn twist along the curve
$c_\al$ of Figure~\ref{curves}~(a). Let $\s_0=\lgl a_0\rgl$. 
As can be seen in the figure, the curve is 
non-separating in a neighborhood of
$S_\al\minus S\cup \del S\cup a_0$.
Let $c_\al'$ be one of the components of $\del S$ appearing in
Figure~\ref{curves}~(a) pushed to the interior of
$S_\al$. Note that $c'_\al$ is also non-separating in the neighborhood.   
Hence the complements of the two curves are diffeomorphic and there exists a diffeomorphism $g$ of the neighborhood fixing its boundaries taking
$c_\al$ to $c'_\al$. Let $\bar g\in \Ga(S_\al)$ be the class of $g$ extended by the identity 
to the whole of $S_\al$. Then $\bar g$ commutes with
the image of $St_{\OO^2}(\s_0)$ in $\Ga(S_\al)$. 

Given $[c]\in H_{k-1}(St_{\OO^2}(\s_0))$, the chain
$c\x t_{c_\al}\x \bar g\in C_{k+1}(\Ga(S_\al))$ has boundary
$d(c\x t_{c_\al}\x \bar g)$

\begin{tabular}{ll}
&$=(-1)^{k+1}c\x t_{c_\al}+d(c\x t_{c_\al})\x \bar g+(-1)^{k}\,\bar g^{-1}(c\x t_{c_\al})\,\bar g$\\
&$=(-1)^{k+1}c\x t_{c_\al}+ (-1)^{k} c\x \bar g+0+ (-1)^{k+1}\,t_{c_\al}^{-1}c\,t_{c_\al}\x \bar g+(-1)^{k} c\x t_{c'_\al}$.
\end{tabular} 
As $t_{c_\al}^{-1}c\, t_{c_\al}=c$ the middle terms cancel and the result
follows from the fact that  $c\x t_{c'_\al}\in \Ga(S)$.
\end{proof}

For the map $\be$, we have a similar situation with $c_\be$ is a curve
in a neighborhood $N$
of $S_\be\minus S\cup \del S\cup a_0$ (see
Figures~\ref{curves}~(b) and \ref{curveb}), but now  
$c_\be$ is separating. 
\begin{figure}[ht]
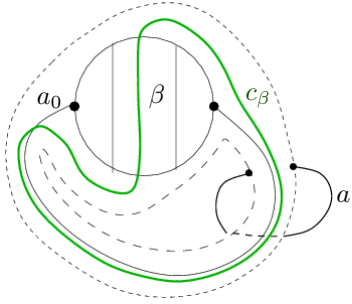

\begin{lpic}{curveb(0.42,0.42)}
 \lbl[b]{48,60;$\be$}
 \lbl[b]{14,60;$a_0$}
 \lbl[b]{80,60;\textcolor{olgreen}{$c_\be$}}
 \lbl[b]{107,30;$a$}
\end{lpic}
\caption{The arc $a$}\label{curveb}
\end{figure}
Consider an arc $a$ in
$S_\be$ as in Figure~\ref{curveb}, joining the two dashed boundary components of $N$  and
otherwise disjoint from it. (Such an arc exists as $a_0$ is non-separating.) Then
$c_\be$ is non-separating in $N\cup a$.

\begin{cor}\label{I3corb}
Let $\s_0$ be a vertex of $\OO^1(S)$. 
Then the composition 
$$ H_{*-1}(\Ga(S'))\rar H_{*-1}(St_{\OO^1}(\s_0))\xrightarrow{\_\x t_{c_\be}} H_*(\Ga(S_\be),\Ga(S))$$
is the zero map, where $S'$ is the complement in $S_\be$ of a neighborhood of
$S_\be\minus S\cup \del S\cup a_0\cup a$, with $a$ as above, and $\Ga(S')\to
St_{\OO^1}(\s_0)$ is the inclusion. 
 \end{cor}

The proof is the same as for Corollary \ref{I3cora} using the fact that $c_\be$ and the component of $\del
S$ in the figure are
non-separating in $S_\be\minus S'$.

\medskip

The last (but not least) ingredient is the connectivity of the complex. The proof is deferred to Section~\ref{connectivity} as 
it does require some work...

\begin{prop}[Ingredient 4]\label{I4}
The complex $\OO^i(S_{g,r})$ is $(g-2)$--connected.
\end{prop}

\section{Spectral sequence argument}\label{SSsect}

This section gives the proof of the main theorem, Theorem~\ref{main}. The proof is a double spectral sequence argument build on the
action of the mapping class group of a surface $S$ on the complexes $\OO^1(S)$ and $\OO^2(S)$. It relies on the geometric results
proved in Section~\ref{ordsect}, and on the connectivity results of Section~\ref{connectivity}.  
We follow the line of argument of \cite{RW09}. 

\medskip

To a simplicial complex $X$, we associate the
augmented chain complex $(\tC_*(X),\del)$ 
defined by $\tC_p(X)=\Z X_p$, the free module over the set of $p$--simplices of $X$, for $p\ge 0$,
and $\tC_{-1}(X)=\Z$. The differential $\del$  induced by the face maps and the augmentation: 
$\del_p=\sum_{i=0}^p(-1)^id_i$ for $p\ge 1$ and $\del_0$ maps the vertices of $X$ to the generator of $\tC_{-1}(X)$.

Let $X,Y$ be simplicial complexes with simplicial actions of groups $G$ and $H$ respectively. 
A homomorphism $\phi\colon G\to H$ together with an equivariant map $f\colon X\to Y$ (with respect to that homomorphism) induces a map of double
chain complexes $$F\colon \tC_*(X)\ot_GE_*G \rar \tC_*(Y)\ot_H E_*H$$
where $(E_*G,d)$ (resp.~$(E_*H,d)$) is a free $G$-- (resp.~$H$--)resolution of $\Z$. 
Now consider the double complex (mapping cone in the $q$-direction)
$$C_{p,q}=(\tC_p(X)\ot_GE_{q-1}G)\oplus(\tC_p(y)\ot_H E_qH)$$
with horizontal differential taking $(a\ot b,a'\ot b')$ to $(\del a\ot b,\del a'\ot b')$ and vertical differential taking 
$(a\ot b,a'\ot b')$ to $(a\ot db,a'\ot db' + F(a\ot b))$. 

The horizontal and vertical filtrations of such a double complex give two spectral sequences, both converging to the homology of the
total complex. 
We will use the following two properties of these spectral sequences: 
\begin{enumerate}[(SS1)]
\item If $X$ is $(c-1)$--connected and $Y$ is $c$--connected, then the $E^1$--term of the horizontal  spectral sequence, 
which is the homology of $C_{p,q}$
  with respect to the horizontal differential, is 0 in the range $p+q\le c$ (noting that $\tC_p(X)$ only contributes to $C_{p,q}$ when
  $q>0$).  
Hence the other spectral sequence converges to 0 in the range $p+q\le c$. 

\item The $E^1$--term of the vertical spectral sequence is the relative homology group 
$E^1_{p,q}=H_q\big(\tC_p(Y)\ot_HE_*H,\tC_p(X)\ot_GE_*G\big)$ as the columns of $C_{p,q}$ are the mapping cones of the map $F$ (with $p$
fixed). 
If the actions of $G$ and $H$ are transitive on $X$ and $Y$, 
a relative version of Shapiro's lemma  identifies this homology group with 
$$E^1_{p,q}=H_q(St_Y(\s_p),St_X(\s_p))$$ 
where $St_X(\s_p)$ and $St_Y(\s_p)$ are the stabilizers in $X$ and $Y$ of some $p$--simplex $\s_p$ of $X$ and its image in $Y$. Note
that this formulation also includes the case $p=-1$ with $St_X(\s_{-1})=G$ and $St_Y(\s_{-1})=H$ as the action is trivial on the
``$(-1)$--simplex''. 
\end{enumerate}

\medskip

Recall from the previous section the maps 
$$\al_g\colon \Ga(S_{g,r+1})\to \Ga(S_{g+1,r}) \ \ \textrm{and}\ \  \be_g\colon \Ga(S_{g,r})\to \Ga(S_{g,r+1}).$$

Denote by $H(\al_g)$ the relative homology group $H(\Ga_{g+1,r},\Ga_{g,r+1};\Z)$ corresponding to the map $\al_g$, and 
$H(\beta_g)$ the relative homology group $H(\Ga_{g,r+1},\Ga_{g,r};\Z)$ corresponding to $\be_g$. 
The main theorem considered in this paper (Theorem~\ref{main}) can be restated as follows:  

\begin{thm}
$(1)\ H_i(\al_g)=0$ for $i\le\frac{2g+1}{3}\ \ $ and  $\ \ (2) \ H_i(\beta_g)=0$ for $i\le \frac{2g}{3}$. 
\end{thm}

\begin{proof}
We prove the theorem by induction on $g$. To start the induction, note that statements $(1)$ for genus 0 and $(2)$ for genus 0,1 
are trivially true as
they are just concerned with $H_0$. Let $(1_g)$ and $(2_g)$ denote the truth of (1) and (2) in the theorem for genus $g$. 
The induction will go in two steps: 
\begin{enumerate}[\hspace{8mm}Step 1:]
\item For $g\ge 1$, $(2_{\le g})$ implies $(1_g)$.
\item For $g\ge 2$, $(1_{<g})$ and $(2_{g-1})$ imply $(2_g)$.
\end{enumerate}

\vs

For Step 1, we consider the spectral sequence described above for the actions of $G=\Ga_{g,r+1}$ on $X=\OO^2(S_{g,r+1})$ and of 
$H=\Ga_{g+1,r}$ on $Y=\OO^1(S_{g+1,r})$ with the homomorphism $\phi\colon G\to H$ and the map $f\colon X\to Y$ both induced by the map
$\al\colon S_{g,r+1}\to S_{g+1,r}$ of Figure~\ref{ab}.  
As the action is transitive in both cases (Propositions~\ref{I1}), we can apply (SS2) from above 
which says that the vertical spectral sequence has the form 
$E^1_{p,q}=H_q(St_Y(\s_p),St_X(\s_p))$. 
When $p=-1$, we have  
$$E^1_{-1,q}=H_q(\Ga_{g+1,r},\Ga_{g,r+1})=H_q(\al_g)$$ which are the groups we are interested in. 
By Propositions~\ref{I2}, the other groups are identified with 
$$E_{p,q}^1=H_q(\be_{g-p}) \ \ \ \ \textrm{for}\ \  p\ge 0.$$
Hence we will be able to apply induction to these terms of the spectral sequence. We want  
to deduce that $E^1_{-1,q}=0$  for $q\le \frac{2g+1}{3}$. 
This will follow from the following three claims: 

\smallskip
\no {\em Claim 1:} 
$E^\infty_{-1,q}=0$ for $q\le \frac{2g+1}{3}$. 

\smallskip 
\no {\em Claim 2:} 
The $E^1$--term is as in Figure \ref{SS1}, i.e. there are no possible sources of differentials to kill classes in
$E^1_{-1,q}$ with $q\le \frac{2g+1}{3}$, except possibly for $d^1\colon E^1_{0,q}\to E^1_{-1,q}$ when $q=\frac{2g+1}{3}$
(i.e.~when the fraction is an integer). 

\smallskip
\no {\em Claim 3:} 
The map $d^1\colon E^1_{0,q}\to E^1_{-1,q}$ is the 0-map.

\begin{figure}[ht]
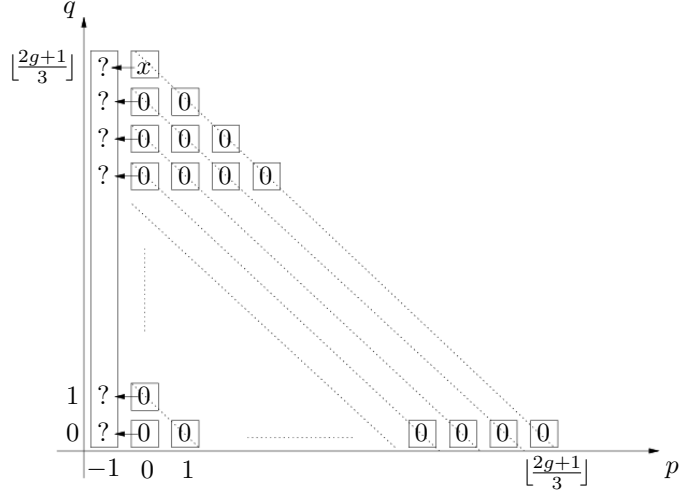

\begin{lpic}{SS.2(0.45,0.45)}
 \lbl[b]{-4,117;$\lfloor\!\frac{2g+1}{3}\!\rfloor$} 
 \lbl[b]{4,136;$q$}
 \lbl[b]{5,22;$1$}
 \lbl[b]{5,11;$0$}
 \lbl[b]{148,-3;$\lfloor\!\frac{2g+1}{3}\!\rfloor$} 
 \lbl[b]{182,0;$p$}
 \lbl[b]{39,0;$1$}
 \lbl[b]{27,0;$0$}
 \lbl[b]{14,0;$-1$}
 \lbl[b]{26,11;$0$}
 \lbl[b]{26,22;$0$}
 \lbl[b]{38,11;$0$}
 \lbl[b]{108,11;$0$}
 \lbl[b]{120,11;$0$}
 \lbl[b]{132,11;$0$}
 \lbl[b]{144,11;$0$}
 \lbl[b]{26,87;$0$}
 \lbl[b]{38,87;$0$}
 \lbl[b]{50,87;$0$}
 \lbl[b]{62,87;$0$}
 \lbl[b]{26,98;$0$}
 \lbl[b]{38,98;$0$}
 \lbl[b]{50,98;$0$}
 \lbl[b]{26,109;$0$}
 \lbl[b]{38,109;$0$}
 \lbl[b]{26,119.5;$x$}
 \lbl[b]{14,119;?}
 \lbl[b]{14,109;?}
 \lbl[b]{14,98;?}
 \lbl[b]{14,87;?}
 \lbl[b]{14,22;?}
 \lbl[b]{14,11;?}
\end{lpic}
\caption{Spectral sequence for Step 1. 
The possible sources of differentials for the ``?'' are along the dotted diagonals.}\label{SS1}
\end{figure}

\smallskip

Claims 1 and 2 imply immediately that $E^1_{-1,q}=0$ for $q<\frac{2g+1}{3}$ as ``it must die by $E^\infty$'' (Claim 1) 
and ``nobody can kill
it'' (Claim 2). Claim 3 gives that this also holds when $q=\frac{2g+1}{3}$ as the only differential with a
possibly non-trivial source is the zero map, and hence won't kill anything in the target.

By Proposition~\ref{I4}, $X$ is $(g-2)$--connected and $Y$ is $(g-1)$--connected. 
Applying (SS1) from above, we get that $E^\infty_{p,q}=0$ for 
$p+q\le g-1$. In particular, $E^\infty_{-1,q}=0$ for $q\le g$. As $\frac{2g+1}{3}\le g$ when $g\ge 1$, Claim 1
follows. 
 
The sources of differentials to $E_{-1,q}^1$ are the terms $E^{p+1}_{p,q-p}$ for $p\ge 0$. As
$E^1_{p,q}=H_{q}(\be_{g-p})$ when $p\ge 0$, by
induction we know that $E^1_{p,q}=0$ when $q\le \frac{2(g-p)}{3}=\frac{2g-2p}{3}$ and $p\ge 0$. 
Hence $E^1_{p,q-p}=0$ for $q\le \frac{2g+p}{3}$ and $p\ge 0$, i.e.~they are all 0 for any $p\ge 0$ if $q\le \frac{2g}{3}$ 
or for $p\ge 1$ if $q=\frac{2g+1}{3}$. This is Claim 2. 

Claim 3 is given by Corollary~\ref{I3cora}: The map $d^1\colon E^1_{0,q}\to E^1_{-1,q}$ is the map 
$H_q(St_{\OO^1}(\al(\s_0)),St_{\OO^2}(\s_0))\to H_q(\Ga(S_\al),\Ga(S))$ 
of the corollary, where $S=S_{g,r+1}$ and $S_\al=S_{g+1,r}$.

\medskip 

For Step 2, the argument is essentially the same. 
We consider the spectral sequence described above for the actions of $G=\Ga_{g,r}$ on $X=\OO^1(S_{g,r})$ and of 
$H=\Ga_{g,r+1}$ on $Y=\OO^2(S_{g,r+1})$ with the homomorphism $\phi\colon G\to H$ and the map $f\colon X\to Y$ both induced by the map
$\beta\colon S_{g,r}\to S_{g,r+1}$.  
We can again apply (SS2) by Proposition~\ref{I1} and get that 
 that the vertical spectral sequence has the form 
$E^1_{p,q}=H_q(St_Y(\s_p),St_X(\s_p))$. 
When $p=-1$, we have  
$$E^1_{-1,q}=H_q(\Ga_{g,r+1},\Ga_{g,r})=H_q(\beta_g)$$ which are the groups we are interested in. 
By Propositions~\ref{I2}, the other groups are identified with 
$$E_{p,q}^1=H_q(\al_{g-p-1}) \ \ \ \ \textrm{for}\ \  p\ge 0$$
Hence we will be able to apply induction to these terms, 
to deduce that $E^1_{-1,q}=0$  for $q\le \frac{2g}{3}$. 
As in the previous case, this follows from three claims: 

\smallskip
\no {\em Claim 1:} 
$E^\infty_{-1,q}=0$ for $q\le \frac{2g}{3}$. 

\smallskip 
\no {\em Claim 2:} 
The $E^1$--term is as in Figure \ref{SS1}, though with $\lfloor \frac{2g+1}{3}\rfloor$ replace by
$\lfloor \frac{2g}{3}\rfloor$, i.e. there are no possible sources of differentials to kill classes in
$E^1_{-1,q}$ with $q\le \frac{2g}{3}$, except possibly for $d^1\colon E^1_{0,q}\to E^1_{-1,q}$ when $q=\frac{2g}{3}$. 

\smallskip
\no {\em Claim 3:} 
The map $d^1\colon E^1_{0,q}\to E^1_{-1,q}$ is the 0-map.

By Proposition~\ref{I4}, $X$ and $Y$ are $(g-2)$--connected. Applying (SS1), we get that $E^\infty_{p,q}=0$ for 
$p+q\le g-2$. In particular, $E^\infty_{-1,q}=0$ for $q\le g-1$. As $\lfloor\frac{2g}{3}\rfloor\le g-1$ when $g\ge 2$, Claim 1
follows. 
 
The sources of differentials to $E_{-1,q}^1$ are the terms $E^{p+1}_{p,q-p}$ for $p\ge 0$. As
$E^1_{p,q}=H_{q}(\al_{g-p-1})$ when $p\ge 0$, by
induction we know that $E^1_{p,q}=0$ when $q\le \frac{2(g-p-1)+1}{3}=\frac{2g-2p-1}{3}$ and $p\ge 0$. 
Hence $E^1_{p,q-p}=0$ for $q\le \frac{2g+p-1}{3}$ and $p\ge 0$, i.e.~they are all 0 for any $p\ge 0$ if $q\le \frac{2g-1}{3}$ 
or for $p\ge 1$ if $q=\frac{2g}{3}$. This is Claim 2. 

Claim 3 is a consequence of Corollary~\ref{I3corb},
 using induction: The map $d^1\colon E^1_{0,q}\to E^1_{-1,q}$ is the map 
$H_q(St_{\OO^2}(\be(\s_0)),St_{\OO^1}(\s_0))\to
H_q(\Ga(S_\be),\Ga(S))$ mapping the top row to the
bottom row of the second square in Proposition~\ref{I3}, where $S=S_{g,r}$ and
$S_\be=S_{g,r+1}$ here. Applying Lemma ~\ref{relative} to the proposition,
we get that this map 
factors through the map 
$$\_ \times t_{c_\be} \colon H_{q-1}(St_{\OO^1}(\s_0))\rar
H_q(\Ga(S_\be),\Ga(S))$$
with $t_{c_\be}$ a Dehn twist along the curve $c_\be$ of Figure~\ref{curves}~(b). 
Let $S'$ be as in Corollary~\ref{I3corb}. By the corollary, the
composition  
$$H_{q-1}(\Ga(S'))\rar 
H_{q-1}(St_{\OO^1}(\s_0))\xrightarrow{\_ \times t_{c_\be}} H_q(\Ga(S_\be),\Ga(S))$$
is the zero map. Note now that the first map is a $\be$-map
$$ H_{q-1}(\Ga(S'))\cong H_{q-1}(\Ga_{g-1,r})\ \rar\ H_{q-1}(\Ga_{g-1,r+1})
\cong H_{q-1}(St_{\OO^1}(\s_0)).$$
By induction, it is surjective: we have
$H_{q-1}(\be_{g-1}) = H_{q-1}(\Ga_{g-1,r+1},\Ga_{g-1,r})=0$ by $(1_{g-1})$ as $q-1\le
\frac{2g}{3}-1\le\frac{2(g-1)}{3}$.
As the composition above is 0, we can deduce that the second map is
also 0 and hence that the differential is 0.  
\end{proof}

Note that the slope $\frac{2}{3}$ in the bound of the stable range is
determined by the structure of the spectral sequence: to prove that 
$H_q(\al_g)=0$ requires that $H_{q-1}(\beta_{g-1})=0$ (for claim 2 of step 1),
which in turn requires that $H_{q-2}(\al_{g-3})=0$ (for claim 2 of step 2).   
Claim 1 does not interfere with the slope because the connectivity
bounds of the complexes used have a higher
slope, and Claim 3 (in Step 2) only requires the slope to be at most 1.
The constant
coefficient is decided by the connectivity of the complexes
for low genus surfaces. Recall though from the introduction that we know the slope to be best
possible, and the constant coefficient to be very close to best possible.

\section{Connectivity argument}\label{connectivity}

In this section, we give the proof of Proposition~\ref{I4} which gives the connectivity of the complexes $\OO^1(S)$ and $\OO^2(S)$ 
used in the previous section to prove the stability theorem.
The rough line of argument is to embed $\OO^i(S)$ in a larger complex which we can show is contractible, and work backwards from there,
deducing high connectivity of smaller and smaller complexes. 
We start by defining the relevant complexes.

\medskip

Let $\Delta\subset \del S$ be a non-empty 
set of points. We consider arcs in $S$ with boundary in $\Delta$. We say that an arc $a$ is
{\em trivial} if it separates $S$ into two components, one of which is a disc intersecting $\Delta$ only in the boundary
of $a$. 
Let $\A(S,\Delta)$ be the simplicial complex whose vertices are isotopy classes of non-trivial arcs in $S$ with boundary in $\Delta$. A
$p$--simplex of $\A(S,\Delta)$ is a collection of $p+1$ distinct isotopy classes of arcs $\lgl a_0,\dots,a_p\rgl$ representable by arcs 
with disjoint interiors. 

The second complex we consider is the complex of arcs between two sets of points. 
Given two disjoint sets of points $\Delta_0,\Delta_1\subset \del S$, define
$\BB(S,\De_0,\De_1)\subset \A(S,\De_0\cup\De_1)$ to be the subcomplex of arcs with one boundary point in $\De_0$ and one
in $\De_1$. 

Let $\BB_{0}(S,\De_0,\De_1)\subset \BB(S,\De_0,\De_1)$ be the subcomplex of non-separating collections, i.e. simplices
$\s=\lgl a_0,\dots,a_p\rgl$ such that
the complement of the arcs $a_0,\dots,a_p$ in $S$ is connected. (This is a subcomplex as the non-separating
property is preserved under taking faces.)

Finally, the complex $\OO(S,b_0,b_1)$ we are interested in is the
subcomplex of\\ 
 $\BB_{0}(S,\{b_0\},\{b_1\})$ of
simplices  $\lgl a_0,\dots,a_p\rgl$ such that the ordering of the arcs at $b_0$ is opposite to that at $b_1$. (See
Definition~\ref{OOdef}.)

\vs

We will prove that $\OO(S,b_0,b_1)$ is $(g-2)$--connected by first proving that $\A(S,\Delta)$ is contractible (in most
cases), and then slowly deducing connectivity bounds for each of the complexes in the sequence 
$$\A(S,\Delta) \sta{i_1}{\hookleftarrow} \BB(S,\Delta_0,\Delta_1)\sta{i_2}{\hookleftarrow}\BB_0(S,\De_0,\De_1)\sta{i_3}{\hookleftarrow}\OO(S,b_0,b_1).$$ 
In the end, we only need the case of $\Delta=\{b_0,b_1\}$, but the connectivity arguments will use an induction
requiring to know the connectivity of complexes with a larger number of points---this comes from the fact that,
cutting the surface along arcs between points of $\Delta$ produces several copies of the original points of $\Delta$ 
(see Figure~\ref{OO.cut}).

The connectivity arguments we will use are of three types: (1) {\em direct calculation} showing contractibility,
(2) exhibition of a complex as a {\em suspension} (or wedge of such)
of a ``previous'' complex, and (3) {\em inductive deduction} from the connectivity of a larger complex. 
The argument for the connectivity of $\A(S,\Delta)$ will be a mix of type (1) and (2), the deduction along $i_1$ in the sequence is the
most intricate argument and will be a mix of the three types of arguments, while deduction along $i_2$ and $i_3$ will be
purely (and simpler) type (3) arguments.

\smallskip

The arguments given in this section are collected from the papers \cite{Har85,Hat91,Iva89,RW09,Wahl08}. 
Theorem~\ref{A}, which gives the contractibility of the full arc complexes, was originally proved by Harer using the
theory of train tracks \cite[Sect.~2]{Har85}. We give here a much simpler proof, by surgering the arcs, 
due to Hatcher \cite{Hat91}. Theorem~\ref{D0D1}, giving the connectivity of the complex $\BB(S,\De_0,\De_1)$,
is also originally due to Harer \cite[Thm.~1.6]{Har85}, though there
is a gap in his proof, which is fixed in \cite[Thm.~2.3]{Wahl08}. We follow here the proof given in \cite{Wahl08} which uses a
mixture of Hatcher's surgery argument, and an careful inductive deduction which we learned from reading
Ivanov \cite{Iva89}. Deducing the connectivity of the non-separating subcomplex $\BB_0(S,\De_0,\De_1)$ (Theorem~\ref{BX}
below) is a more standard argument that goes back at least to Harer. Finally Theorem~\ref{OS} 
gives the
connectivity of the ordered complex $\OO(S,b_0,b_1)$. This was obtained by Ivanov \cite[Thm.~2.10]{Iva89} 
in the case where the two points are
on the same boundary component via a different sequence of complexes and where the ordering comes naturally from a
Morse-theoretic argument. The general case is given by Randal-Williams in \cite[Thm.~A.1]{RW09}, deducing it from Theorem~\ref{BX} 
via a combinatorial argument, similar to that of Theorem~\ref{BX}. 

\smallskip

We start by proving the contractibility of the full arc complex $\A(S,\De)$:

\begin{thm}\label{A}
$\A(S,\Delta)$ is contractible, unless $S$ is a disc or an annulus with $\Delta$ included in a single component of $\del S$, in which
case it is $(q+2r-7)$--connected, where $q=|\Delta|$ and $r=1,2$ is the number of boundary components of $S$.  
\end{thm}

Note that even though we are mostly interested in surfaces of positive genus,
the bound for the connectivity of $\A(S,\De)$ in the case of discs will actually play a role in the proof of the
connectivity of the next complex, $\BB(S,\De_0,\De_1)$. 

We start by a suspension lemma, i.e. type (2) argument: 

\begin{lem}\label{lemA}
Suppose $\A(S,\Delta)\neq \emp$ and $\Delta'$ is obtained from $\Delta$ by adding an extra point in a component of $\del S$ already
containing a point of $\Delta$. 
If $\A(S,\Delta)$ is $d$--connected, then  $\A(S,\Delta')$ is $(d+1)$--connected. 
\end{lem}

One can actually show that $\A(S,\Delta')$ is homeomorphic to the suspension of $\A(S,\Delta)$ (see \cite{Hat91}, revised version) 
but we will not need that.

\begin{proof}
\begin{figure}[ht]
\begin{lpic}{II.2(0.50,0.50)}
 \lbl[b]{13,16;$p$}
 \lbl[b]{29,15;$p'$}
 \lbl[b]{42,2;$I$}
 \lbl[b]{-1,2;$I'$} 
\end{lpic}
\caption{}\label{II}
\end{figure}
Suppose $\Delta'=\Delta\cup \{p'\}$ and $p\in\Delta$ is a closest element to $p'$ in $\del S$. Let $I$ and $I'$ be the
arcs drawn in Figure~\ref{II}, where the points of $\Delta'$ to the left of $p$ and right of $p'$ may be equal to $p'$ or $p$. 
We have a decomposition 
$$\A(S,\Delta')=\Star(I)\cup_{\link(I)}X$$ 
where $X$ is the subcomplex of $\A(S,\Delta')$ of collections of arcs not containing $I$. We will show that $X$ deformation retracts
onto the star of $I'$, and hence that it is contractible.  
The result then follows from the fact
that the link of $I$ is isomorphic to $\A(S,\De)$. (For $d\ge 0$, van Kampen's theorem implies that $\A(S,\De')$ is simply connected,
so that it is enough to check connectivity in homology, which follows from 
the Mayer-Vietoris exact sequence.) 

$\Star(I')$ is exactly the subcomplex of $X$ of arcs without endpoints at $p$. 
The idea of the retraction $X\to\Star(I')$ is to move the arcs one by one from $p$ to $p'$ along $I'$, as shown in Figure~\ref{retraction}. 
\begin{figure}[ht]
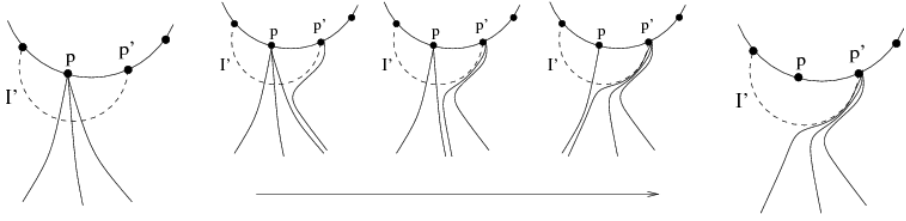

\begin{lpic}{retract.flip(0.50,0.50)}
\end{lpic}
\caption{Simplices $r_1,r_2,r_3$ of the retraction in the case of $3$ germs of arcs at $p$}\label{retraction}
\end{figure}

Note that the only arc that would become trivial when slided from $p$ to $p'$ is $I$, and it is not in $X$. 

The retraction can be made explicit as follows. We want a map $f\colon I\x X\to X$ so that $f(s,\Star(I'))=id_{\Star(I')}$ and
$f(1,X)\subset \Star(I')$. Given a simplex $\s=\lgl a_0,\dots,a_q\rgl$ of $X$, 
we can consider the arcs of $\s$ attached at
$p$. Suppose that there are $k$ 'germs of arcs' $\ga_{1},\dots,\ga_{k}$ occurring in that order at $p$, 
with $\ga_i$ a germ of the arc $a_{j_i}$, where it is possible that $j_i=j_{i'}$ for some $i\neq i'$ if the arc $a_{j_i}$ has both its
endpoints at $p$. 
There is a sequence of $k$ $(q+1)$--simplices $r_1(\s),\dots,r_k(\s)$ associated to $\s$, where $r_i(\s)$ is obtained from $\s$ by
moving the first $i$ germs of arcs at $p$ to $p'$ and keeping the last $k-i+1$ germs, so that $\ga_i$ has a copy both at $p$ and at
$p'$. (See Figure~\ref{retraction} for an example when $k=3$.) 
If $L$ denotes the operator on arcs that moves the first germ of the arc at $p$ to $p'$, we have $r_i(\s)=\lgl b_0,\dots,b_{q+1}\rgl$ 
with $b_l=L^{\epsilon_i(l)}(a_l)$ for $l\le q$ and $b_{q+1}=L^{\epsilon_i(j_i)+1}(a_{j_i})=L^{\epsilon_{i+1}(j_i)}(a_{j_i})$, 
where $\epsilon_i(l)$ is
the number of $j<i$ such that $\ga_j$ is a germ of $a_l$.

A point in a simplex $\s$ corresponds to a weighted collection of arcs via the barycentric coordinates $(t_0,\dots,t_q)$, 
the arc $a_i$ having weight $t_i$.  
Assign to the $i$th germ $\ga_i$ the weight $w_i=t_{j_i}/2$. As $\sum_{j=0}^pt_j=1$, we have $\sum_{i=1}^kw_i\le 1$.   
Now for $\sum_{j=1}^{i-1}w_j\le s\le \sum_{j=1}^iw_j$, define the retraction by 
$$f(s,[\s,(t_0,\dots,t_q)])=[r_i(\s),(v_0,\dots v_{q+1})]$$ 
where the weight $v_i=t_i$ except for  the pair
$$(v_{j_i},v_{q+1})=(t_{j_i}-2(s-\sum_{j=1}^{i-1}w_j),2(s- \sum_{j=1}^{i-1}w_j))$$ 
i.e. the weight of $(b_{j_i},b_{q+1})$ goes from $(t_{j_i},0)$ to
$(0,t_{j_i})$ as $s$ goes from  $\sum_{j=1}^{i-1}w_j$ to  $\sum_{j=1}^iw_j$.
For $\sum_{i=1}^k w_i\le s\le 1$, define  $f(s,(\s,(t_0,\dots,t_q)))$ to be constant, equal to 
$f(\sum_{i=1}^k w_i,(\s,(t_0,\dots,t_q)))$.  
Note that  $f(1,(\s,(t_0,\dots,t_q)))$ lies in the face of $r_k$ which is in $\Star(I')$. 

This deformation is continuous as going
to a face of $\s$ corresponds to a $t_i$ (and the corresponding $w_j$ if any) going to zero.
\end{proof}

\begin{proof}[Proof of Theorem~\ref{A}]
We first consider the special case of the disc and cylinder with all points of $\Delta$ in one boundary component of $S$.  
Figure~\ref{nonemp}(a) shows that $\A(D^2,\Delta)$ is non-empty as soon as $\Delta$ has $4$ points, and \ref{nonemp}(b) that 
$\A(S^1\x I,\Delta)$ is non-empty if
$S^1\x I$ has two points in one boundary component. As $4+2-7=-1$ ($q=4, r=1$) and $2+4-7=-1$ ($q,r=2$), and $(-1)$--connected means
non-empty, 
the theorem, for the special cases, is true when $r=1$ and $q\le 4$, and when $r=2$ and $q\le 2$. 
The result then follows more generally for any $q$ by the lemma, which shows that the connectivity of these complexes goes up by one each time $q$ goes up by one.
\begin{figure}[ht]
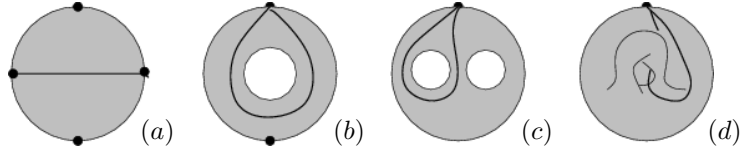

\begin{lpic}{nonemp(0.52,0.52)}
 \lbl[b]{38,0;$(a)$}
 \lbl[b]{87,0;$(b)$}
 \lbl[b]{135,0;$(c)$}
 \lbl[b]{183,0;$(d)$} 
\end{lpic}
\caption{Checking non-emptiness for the disc, cylinder, pair of pants and genus 1 surface}\label{nonemp}
\end{figure}
We assume from now on that we are not in the special cases.

For the general case, the lemma allows us to assume that there is at most one point of $\Delta$ in each boundary component. We claim
first that $\A(S,\Delta)$ is non-empty: this is clear if $\Delta$ has at least 2 points, as they lie in different boundary
components---any arc connecting two such points will be non-trivial. 
If $\Delta$ has only one point, we have that $S$ has genus at least one or has at least three boundary components.
Figure~\ref{nonemp}(c) and (d) show that there is at least one non-trivial arc in both of these cases. 

Now fix a point $p$ of $\Delta$ and an arc $a$ of $\A(S,\Delta)$ with at least one of its endpoints at $p$. Fix also a germ of $a$ at
$p$. 
We want to define a retraction of the complex onto the star of $a$. The argument is
similar to that of the proof of Lemma~\ref{lemA}, and is summarized by Figure~\ref{retract2}. 
\begin{figure}[ht]
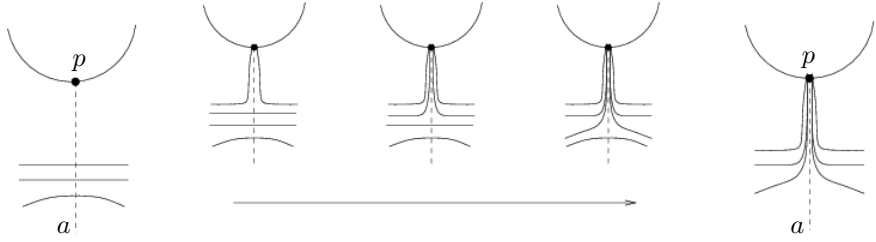

\begin{lpic}{retract2(0.50,0.50)}
 \lbl[b]{19,43;$p$} 
 \lbl[b]{15,0;$a$} 
 \lbl[b]{213,44;$p$} 
 \lbl[b]{210,0;$a$} 
\end{lpic}
\caption{Retraction of $\A(S,\De)$ in the case of $3$ germs of arcs crossing $a$}\label{retract2}
\end{figure}
Given a $q$--simplex $\s$ of $\A(S,\Delta)$, we
represent it by a simplex with minimal and transverse intersection with $a$. If there are $k$ intersections at germs of arcs
$\ga_1,\dots,\ga_k$, we define as in the lemma $k$ intermediate simplices $r_1(\s),\dots,r_k(\s)$ by successively cutting the arcs at the
intersection points and connecting the new endpoints to $p$ along $a$. 

More precisely, 
if $a_i$ intersects $a$ at a point $x$, we can define $L(a_i)$ and $R(a_i)$ to be the arcs obtained from $a_i$ by
cutting $a_i$ at $x$
and joining the new endpoints to $p$ along $a$, then pushing the arcs a little
so they become disjoint of $a$, of $a_i$ and of each other---there is then one
arc on each side of $a$ between $x$ and $p$ and we call $L(a_i)$
(resp.~$R(a_i)$) the arc running to the left (resp.~right) of $a$ towards $p$. 
If $L(a_i)$ or $R(a_i)$ is a trivial arc, we forget
it. As there is at most one point of $\De$ per boundary components, and arcs intersect $a$ minimally, 
the only trivial arcs that can occur in this
process are of the type shown in Figure~\ref{trivial} where one side of the surgered arc becomes trivial. 
In particular both $L(a_i)$ and $R(a_i)$ cannot be trivial at the same
time.
\begin{figure}[ht]
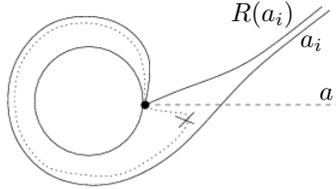

\begin{lpic}{trivial.2(0.50,0.50)}
 \lbl[b]{85,23;$a$}
 \lbl[b]{81,36;$a_i$}
 \lbl[b]{68,42;$R(a_i)$} 
\end{lpic}
\caption{Arc where surgery creates one trivial arc (and one non-trivial one)}\label{trivial}
\end{figure}

Define $r_i(\s)$ to be the $q+l_i$--simplex, with  $1\le l_i\le i+2$, 
with vertices $L^{\epsilon_i(l)}(a_l)$ and $R^{\epsilon_i(l)}(a_l)$ (if non-trivial), where $\epsilon_i(l)\in\{0,\dots,k\}$ 
is the number germs $\ga_j$ included in $a_l$ with $j<i$, and where
$L^{0}(a_l)=R^{0}(a_l)=a_l$ is a single vertex. 
Then the retraction is as in the lemma though when
$a_{j_i}$ is replaced by two arcs $R(a_{j_i})$ and $L(a_{j_i})$, the weight of $(a_{j_i},R(a_{j_i}),L(a_{j_i}))$ in $r_i$ will go from
$(t_{j_i},0,0)$ to $(0,t_{j_i}/2,t_{j_i}/2)$ so that the total weight stays equal to 1. 

For this retraction to be well-defined, we need that the arcs obtained by cutting do not depend on the representative of $\s$. This
follows from the fact that isotopic collections of arcs in minimal transverse intersection are isotopic through minimal transverse
intersection. This can be proved by modifying a given an isotopy  as in the proof of
Proposition~\ref{I1}. 
\end{proof}

Next we deduce the connectivity of the subcomplex $\BB(S,\Delta_0,\Delta_1)$ of arcs between two subsets $\Delta_0$ and $\Delta_1$ of
$\Delta$.  This is the most intricate of all the connectivity argument. 
To begin with, we need a definition to be able to state the connectivity bound: Disjoint sets $\De_0,\De_1\subset \del S$ define a decomposition of
$\del S$ into vertices (the points of $\De_0\cup\De_1$), edges between the vertices, and circles without vertices.
 We say that an edge is {\em pure} if both
its endpoints are in the same set, $\De_0$ or $\De_1$. We say that an edge is {\em impure} otherwise. Note that the number of impure edges
is always even.

\begin{thm}\label{D0D1}
The complex $\BB(S,\Delta_0,\Delta_1)$ is $(4g+r+r'+l+m-6)$--connected, where $g$ is the genus of $S$, $r$ its number of boundary
components, $r'$ the number of components of $\del S$ containing points of $\Delta_0\cup\Delta_1$, $l$ is half the
number of impure edges and $m$ is the number
of pure edges. 
\end{thm}

The proof of the theorem is in the same spirit as that of the previous theorem: the general argument works only in a restricted
situation, so one first eliminates a number of cases, which we do by `pushing' and surgery arguments as in the lemma and the theorem
above. The argument for the general case will then be our first inductive deduction (argument of type (3)).

We say that a boundary component of $S$ with points of $\De_0\cup\De_1$ is {\em pure} if it is composed of pure edges,
 i.e.~the points are either all in 
 $\De_0$ or all in $\De_1$.
The first lemma gives contractibility when $S$ has a pure boundary, which is the case where the argument of Theorem~\ref{A} can be applied: 

 \begin{lem}\label{case1}
 If $S$ has at least one pure boundary component, then $\BB(S,\De_0,\De_1)$ is contractible. 
\end{lem}

\begin{figure}[ht]
\begin{lpic}{lemmas12.2(0.58,0.58)}
 \lbl[b]{90,0;(a)}
 \lbl[b]{10,27;$a$}
 \lbl[b]{33,10;$a_i$}
 \lbl[b]{47,9;$R(a_i)$}
 \lbl[b]{47,30;$L(a_i)$}
 \lbl[b]{198,0;(b)} 
 \lbl[b]{157,8;$I'$}
\end{lpic}
\caption{Lemmas~\ref{case1} and \ref{case2}}\label{lemmas12}
\end{figure}

\begin{proof}
Choose an arc $a$ with one boundary point on a pure boundary of $S$. We do a retraction onto the star of $a$ as in the proof of
Theorem~\ref{A}, though where there is always exactly one arc which is kept after surgery as only one of $L(a_i)$ and $R(a_i)$ still
has a boundary point in each of $\De_0$ and $\De_1$. (See Figure~\ref{lemmas12}(a).) The retraction is well-defined because each newly
created arc will necessarily be non-trivial, having its boundary points in two different components of $\del S$. 
\end{proof}

\begin{lem}\label{case2}
If $S$  has at least one pure  edge between a pure and an impure one in a  boundary component of $S$, 
then $\BB(S,\De_0,\De_1)$ is contractible. 
\end{lem}

\begin{proof}
This is completely analogous to the proof of Lemma~\ref{lemA}, except that the arc corresponding to $I$ in the case of
Lemma~\ref{lemA} does not exist in our complex $\BB(S,\De_0,\De_1)$ (compare Figure~\ref{II} with 
Figure~\ref{lemmas12}(b)). Hence the argument of Lemma~\ref{lemA} gives a contraction 
from $X=\BB(S,\De_0,\De_1)$ to the star of $I'$, showing that the complex is contractible. 
\end{proof}

\begin{lem}\label{case3}
If the complex is non-empty, adding a pure edge between two impure edges increases the connectivity of $\BB(S,\De_0,\De_1)$ by one. 
\end{lem}

\begin{figure}[ht]
\begin{lpic}{lemmas34.2(0.58,0.58)}
 \lbl[b]{57,0;(a)}
 \lbl[b]{30,43;$I$}
 \lbl[b]{25,3;$I'$} 
 \lbl[b]{160,0;(b)}
 \lbl[b]{113,45;$I_0$}
 \lbl[b]{141,29;$\del_0S$} 
\end{lpic}
\caption{Lemmas~\ref{case3} and \ref{case4}}\label{lemmas34}
\end{figure}

\begin{proof}
This is now precisely as in Lemma~\ref{lemA}, as shown in Figure~\ref{lemmas34}(a). 
\end{proof}

\begin{lem}\label{case4}
When $S$ has at least one impure edge and the complex is non-empty, adding a boundary component to $S$ disjoint from $\De$ increases the connectivity of $\BB(S,\De_0,\De_1)$ by one. 
\end{lem}

\begin{proof}
This is a variation on the proof of Lemma~\ref{lemA}. Suppose $\del_0 S$ is a boundary component of $S$ disjoint from $\Delta$, and let
$S'$ be $S\cup_{\del_0S} D^2$. We want to show that $\F(S,\De_0,\De_1)$ is one more connected than $\F(S',\De_0,\De_1)$.
 
Call an arc $I$ of $\F(S,\De_0,\De_1)$ {\em special} if it separates a cylinder from $S$, one of whose boundary components is $\del_0S$ and
the other one intersects $\De_0\cup\De_1$ only at the endpoints of $I$. As any two distinct special arcs must intersect, we have 
$$\F(S,\De_0,\De_1)=X\bigcup_{\begin{subarray}{l}\link(I),\\ I \ \textrm{special}\end{subarray}}\Star(I)$$
where $X$ is the subcomplex of $\F(S,\De_0,\De_1)$ of simplices having no special arc among their vertices. 
Note that the link of any special arc $I$ is isomorphic to $\F(S',\De_0,\De_1)$. 

Pick a special arc $I_0$. We can produce a retraction of $X\cup\Star(I_0)$ onto $\Star(I_0)$. As in the proof of Lemma~\ref{lemA}, we
do this by producing, for any $p$--simplex $\s$ intersecting $I_0$, a sequence of $(p+1)$--simplices $r_1,\dots,r_k$
obtained by moving the intersections of $\s$ with $I_0$ one by one across $\del_0S$ in the way shown in  Figure~\ref{lemmas34}(b).
Passing an arc across $\del_0$ creates a trivial arc only in the case of special arcs (different from $I_0$), 
and these are assumed not to be in $X$. Hence the retraction is well-defined and result follows from van Kampen's
theorem and the Mayer-Vietoris long exact sequence
for the decomposition 
$\F(S,\De_0,\De_1)=\big(X\cup \Star(I_0)\big)\,\bigcup_{\begin{subarray}{l}\link(I),\\ I \ \textrm{special}, I\neq I_0\end{subarray}}\Star(I)$. 
\end{proof}

\begin{figure}[ht]
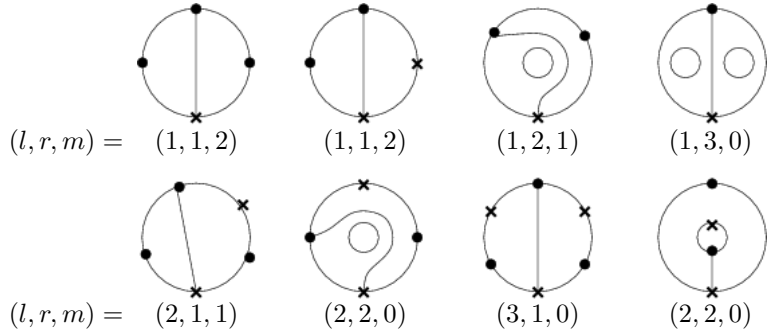

\begin{lpic}{specialcases(0.56,0.56)}
 \lbl[b]{-17,34;$(l,r,m)=$} 
 \lbl[b]{14,34;$(1,1,2)$} 
 \lbl[b]{54,34;$(1,1,2)$} 
 \lbl[b]{95,34;$(1,2,1)$} 
 \lbl[b]{136,34;$(1,3,0)$} 
 \lbl[b]{-17,-7;$(l,r,m)=$}
 \lbl[b]{14,-7;$(2,1,1)$} 
 \lbl[b]{54,-7;$(2,2,0)$}
 \lbl[b]{95,-7;$(3,1,0)$}
 \lbl[b]{136,-7;$(2,2,0)$}    
\end{lpic}
\caption{Verifying non-emptiness}\label{specialcases}
\end{figure}

\begin{proof}[Proof of Theorem~\ref{D0D1}] 
Note first that the theorem is obviously true in the cases where we have already 
shown that $\BB(S,\De_0,\De_1)$ is contractible. The proof will inductively deduce the connectivity of $\BB(S,\De_0,\De_1)$ 
from that of $\A(S,\De_0\cup\De_1)$ in the
cases not already taken care of by the lemmas. 

For $g=0$ and $r'=1$, according to the theorem, non-emptiness occurs when $r+l+m\ge 4$. As $r'=1$,  we have $l\ge 1$. 
If $l=1$, this means either $r=1$ and $m\ge 2$,  $r=2$ and $m\ge 1$, or $r\ge 3$. If $l=2$, we need either $r=1$ and $m\ge 1$ or $r\ge
2$.  
Non-emptiness in these five cases are verified in Figure~\ref{specialcases}.
The result then follows more generally for the case $g=0,r'=1$ and $l=1,2$ by Lemmas~\ref{case2},\ref{case3} 
to change the value of $m$ and Lemma~\ref{case4} to
change the value of $r$. 

We will prove the result in general by induction on the lexicographically ordered triple $(g,r,q)$, where the genus $g\ge 0$, 
the number of boundaries $r\ge 1$ 
 and $q=2l+m\ge 2$ is the cardinality of $\Delta=\De_0\cup\De_1$.
By the above, we can assume $l\ge 3$ if $g=0$ and $r'=1$. 
For each $(g,q,r)$, it is enough to show the result in the case $r=r'$ (by Lemma~\ref{case4}) and $m=0$ (by Lemmas~\ref{case1},\ref{case2},\ref{case3}), 
as $\BB(S,\De_0,\De_1)$ is non-empty when $r'=1$ and $l\ge 3$ or when $r'\ge 2$ as shown in Figure~\ref{specialcases}. 
The induction starts with 
 $(g,r,q)=(0,1,6)$, where $q=2l$ as $m=0$ (which is non-empty as already checked).

So consider a surface $S$ and a pair of sets $(\De_0,\De_1)$ in $\del S$ with $r=r'$ and  $m=0$ (and $l\ge 3$ if $g=0$ and $r=1$). 
We need to show that $\BB(S,\De_0,\De_1)$ is $(4g+2r+l-6)$--connected.
Let $k\le 4g+2r+l-6$ and 
 $f\colon S^k\to \BB(S,\De_0,\De_1)$ be a map, which we may assume to be simplicial for some PL triangulation of $S^k$ by Theorem~\ref{approx}. 
By Theorem~\ref{A} there is an extension $\hat f$ of $f$ to $\A(S,\De)$: 
$$\xymatrix{S^k \ar[r]^-f \ar@{^{(}->}[d] & \BB(S,\De_0,\De_1) \ar@{^{(}->}[d] \\
D^{k+1} \ar[r]^-{\hat f} & \A(S,\De)
}$$  
Indeed, $\A(S,\De)$ is contractible, unless $g=0$ and $r'=1$ (with $r=r'$ here) in which case we need $4g+2r+l-6\le 2r+q-7$ 
which is clear as $q=2l\ge 6$ in this case. Using Theorem~\ref{approx} again, 
we may moreover assume that $\hat f$ is simplicial with respect to some PL triangulation of $D^{k+1}$ extending the triangulation of $S^k$. 
We are going to inductively deform $\hat f$  so that its image lies in $\BB(S,\De_0,\De_1)$. 

For a simplex $\s$ of $D^{k+1}$, we say that $\s$ is {\em bad} if its image lies in the complement of $\BB(S,\De_0,\De_1)$ in
$\A(S,\De)$,  i.e. if all the vertices of $\s$ are mapped to {\em pure arcs}, arcs not in $\BB(S,\De_0,\De_1)$.
 Cutting $S$ along the arcs of $\hat f(\s)$, we have a decomposition into connected components 
$$(S,\De_0,\De_1)\minus \hat f(\s)=(X_1,\De_0^1,\De^1_1)\sqcup\dots\sqcup (X_c,\De_0^c,\De_1^c)\sqcup (Y_1,\Ga_1)\sqcup \dots\sqcup (Y_d,\Ga_d)$$
where each $\De^i_\epsilon$ is a non-empty set in $\del S$ inherited from $\De_\epsilon$, with $\epsilon=0,1$, and
$\Ga_j$ is a set containing copies of points of 
either $\De_0$ or $\De_1$, i.e. the $Y_j$'s have only points of one type. Here by ``inherited'', we mean that there is a copy of a
point of $\De$ in each $X_i$ or $Y_j$ neighboring it (as in Figure~\ref{OO.cut}).

Let $Y_\s=i_1(Y_1)\cup \dots\cup i_d(Y_d)$, where $i_j\colon Y_j\to S$ is the canonical inclusion ---which is not necessarily injective on $\del Y_j$.  
Note that each component of $Y_\s$ has only points of one type in its boundary. 
We say that $\s$ is {\em regular} if no arc of $\hat f(\s)$ lies inside $Y_\s$. When $\s$ is regular, $Y_\s$ is the disjoint union of the
$Y_j$'s modulo the identification of 
the points $\Ga_j$'s coming from the same point of $\De$. 

Let $\s$ be a regular bad $p$--simplex of $D^{k+1}$ maximal with respect to the ordered pair $(Y_\s,p)$, where 
$(Y_{\s'},p')< (Y_\s,p)$ if $Y_{\s'}\varsubsetneq Y_\s$ with $\del Y_{\s'}\minus \del Y_\s$ a union of non-trivial arcs in $Y_\s$, or if $Y_{\s'}=Y_\s$ and $p'<p$. 
The map $\hat f$ restricts on the link of such a simplex $\s$ to a map 
$${\rm Link}(\s) \to J_\s=\BB(X_1,\De^1_0,\De_1^1)*\dots * \BB(X_c,\De^c_0,\De_1^c)* \A(Y_1,\Ga_1)* \dots * \A(Y_d,\Ga_d).$$
Indeed, if a simplex $\tau$ in the link of $\s$ maps to a simplex of pure arcs of $X_i$ for some $i$, then $Y_\s\subset Y_{\tau *\s}$ and either 
$Y_{\tau *\s}>Y_\s$, or $Y_{\tau *\s}=Y_\s$ and $\s$ was not of maximal dimension. (On the other hand arcs of $\A(Y_j,\Ga_j)$ 
could not be added to $\s$ by regularity.)

As the triangulation of $D^{k+1}$ is a PL triangulation, we have  ${\rm Link}(\s)\cong S^{k-p}$ (see Appendix). 
We will show now that $J_\s$ is at least $(k-p)$--connected. 

In the case when one of the $Y_j$'s has non-zero genus or at least two boundary components, $\A(Y_j,\Ga_j)$ is
contractible by Theorem~\ref{A} (as each boundary of $Y_j$ has points of $\Ga_j$) and hence so is
$J_\s$ by Proposition~\ref{joinconn}.  
So we are left to consider the case that all the $Y_j$'s have genus 0 and one boundary, i.e. $Y_j$ is a disc for each $j$.  

We have $\chi(S\minus \hat f(\s))=\chi(S)+p'+1$, where $p'+1\le p+1$ is the number of distinct arcs in the image of $\s$. 
If $X_i$ has genus $g_i$ and $r_i=r'_i$ boundaries, the above equation gives 
$\sum_{i=1}^c (2-2g_i - r_i) + d= 2-2g-r+p'+1$ or equivalently 
$$\textstyle{\sum_{i=1}^c} (2g_i + r_i)=2g+r-p'+2c+d-3.$$
Moreover, we have 
$$\textstyle{\sum_{i=1}^{c}}m_i+\textstyle{\sum_{j=1}^d} q_j=2p'+2,$$ 
where $m_i$ is the number of pure edges in $X_i$ and $q_j=|\Ga_j|$,  and $\sum_{i=1}^{c}l_i=l$,
where $l_i$ is half the number of impure edges in $X_i$.

By induction, $\BB(X_i,\De_0^i,\De_1^i)$ is $(4g_i+2r_i+l_i+m_i-6)$--connected. Indeed for each $i$, we have 
$(g_i,r_i,q_i)<(g,r,q)$.  
On the other hand, $\A(Y_j,\Ga_j)$ is  $(q_j -5)$--connected by Theorem~\ref{A}.  
Applying Proposition~\ref{joinconn},  we get \\
${\rm Conn}(J_\s)\ge \sum_{i=1}^c (4g_i + 2r_i + l_i + m_i-4) + \sum_{j=1}^d (q_j -3)-2$

$=4g+2r-2p'+4c+2d-6+l+2p'+2-4c-3d-2$

$=4g+2r+l-d-6\ge 4g+2r+l-p-6\ge k-p$\\
because $3d\le p+1$, and hence $d\le p$, 
as the edges of $Y_j$ are arcs of $\s$ (as we assumed $m=0$), 
minimum three edges are needed for a non-trivial $Y_j$ and the same edge cannot be used  
twice in the $Y_j$'s by regularity of $\s$. 

Hence in all cases, there exists a PL $(k-p+1)$--disc $K$ with $\del K=\link(\s)$, and a simplicial map
$F\colon K\to J_\s\inc\A(S,\Delta)$.   

Now we have $\Star(\s)=\s*\link(\s)$ is a $(k+1)$--disc with boundary $\del\s*\link(\s)$ (see comments after Lemma~\ref{D*S}). In the triangulation of
$D^{k+1}$, we replace that disc with the disc $\del\s* K$ which has same boundary, and we 
modify $\hat f$ in the interior of the disc using the map 
$$\hat f * F\ \colon \ \del \s * K \ \rar\  \A(S,\De)$$
which agrees with $\hat f$ on $\del(\del\s*K)=\del\s*\link(\s)$.   
We are left to show that we have improved the 
situation this way. The new simplices are of the form $\tau=\al * \beta$ with $\al$ a proper face of $\s$ and $\beta$ mapping to $J_\s$. Suppose $\tau$ is 
a regular bad simplex. Then each arc of $\beta$ is pure and hence in $\A(Y_j,\Ga_j)$ for some $j$. Thus $Y_\tau\subseteq Y_\s$. If they are 
equal, we must have $\tau=\alpha$ is a face of $\s$ (by regularity of $\tau$ and $\s$). 
So $(Y_\tau,\dim(\tau))<(Y_\s,p)$. Hence we have reduced the number of regular bad simplices of maximal dimension. As
any bad simplex contains a regular bad subsimplex, the result follows by induction. 
\end{proof}

We use now the connectivity of $\BB(S,\De_0,\De_1)$ to deduce that of the subcomplex $\BB_0(S,\De_0,\De_1)$ of
non-separating simplices.

\begin{thm}\label{BX}
If $\De_0,\De_1$ are two disjoint non-empty sets of points in $\del S$, then the complex
 $\BB_0(S,\De_0,\De_1)$ is $(2g+r'-3)$--connected, for $g$ and $r'$ as above.
\end{thm}

\begin{proof}
We prove the theorem as the previous one by induction on the lexicographically ordered triple $(g,r,q)$, 
where $r\ge r'$ is the number of components of $\del S$ and $q=|\De_0\cup
\De_1|\ge 2$.  
To start the induction, note that the theorem is true when $g=0$ and $r'\le 2$ for any $r\ge r'$ and any $q$, and 
more generally that the complex is non-empty whenever $r'\ge 2$ or $g\ge 1$.

So fix $(S,\De_0,\De_1)$ satisfying $(g,r,q)\ge(0,3,2)$. 
Then $2g+r'-3\le 4g+r+r'+l+m-6$. Indeed, $r\ge 1$ and $l+m\ge 1$. Moreover we assumed that either $r\ge 3$ or $g\ge 1$. 

Let $k\le 2g+r'-3$ and consider a map $f\colon S^k\to \BB_0(S,\De_0,\De_1)$, which we may assume to be simplicial for some PL triangulation of $S^k$ (by
Theorem~\ref{approx}).  This map can be extended to a simplicial map $\hat f\colon D^{k+1}\to \BB(S,\De_0,\De_1)$ by 
Theorems~\ref{D0D1} and the above calculation, for a PL triangulation of $D^{k+1}$ extending that of $S^k$, using again Theorem~\ref{approx}. 
We call a simplex $\s$ of $D^{k+1}$ {\em regular bad} if $\hat f(\s)=\lgl a_0,\dots,a_p\rgl$ and each $a_j$ separates
$S\minus(a_0\cup\dots\widehat{a_j}\dots\cup a_p)$. 
Let $\s$ be a regular bad simplex of maximal dimension $p$. Write $S\minus\hat f(\s)=X_1\sqcup\dots\sqcup X_c$ with each $X_i$ connected. 
By maximality of $\s$, $\hat f$ restricts to a map 
$${\rm Link}(\s)\rar J_\s=\BB_0(X_1,\De_0^1,\De_1^1)*\dots * \BB_0(X_c,\De_0^c,\De_1^c)$$
where each $\De^i_\epsilon$  is inherited from $\De_\epsilon$ and is non-empty as the arcs of $\hat f(\s)$ are impure. Each $X_i$  has
$(g_i,r_i,q_i)<(g,r,q)$, so by induction 
$\BB_0(X_i,\De_0^i,\De_1^i)$ is $(2g_i+r'_i-3)$--connected. The Euler characteristic gives $\sum_i(2-2g_i-r_i)=2-2g-r+p'+1$, where $p'+1\le p+1$ is the
number of arcs in $\hat f(\s)$. We also have $\sum_i(r_i-r'_i)=r-r'$, so $\sum (2g_i+r'_i)=2g+r'-p'+2c-3$. 
Now  $J_\s$ is 
$(\sum_i(2g_i+r'_i-1)-2)$--connected (using Proposition~\ref{joinconn}), that is $(2g+r'-p'+c-5)$--connected. As $c\ge 2$ and $p'\le p$, 
we can extend the restriction of $\hat f$ to $\link({\s})\simeq S^{k-p}$ to a map $F\colon K\to J_\s$ with $K$ a $(k-p+1)$--disc with boundary
the link of $\s$. 
We modify $\hat f$ on the interior of the star of $\s$ using $\hat f * F$  on $\del \s * K\simeq \Star(\s)$ as in the proof of Theorem~\ref{D0D1}. 
If a simplex $\al * \beta$ in $\del\s * K$ is regular bad, $\beta$ must be trivial since $\beta$ does not separate $S\minus\hat f(\al)$, so that $\al * \beta=\al$ is a face of $\s$. We have thus  reduced the number of regular bad simplices of maximal dimension and the result follows by induction. 
\end{proof}

We are now (finally!) ready to prove that the ordered complex $\OO(S,b_0,b_1)$ has the connectivity claimed: 

\begin{thm}[Proposition~\ref{I4}]\label{OS}
$\OO(S,b_0,b_1)$ is $(g-2)$--connected. 
\end{thm}

\begin{proof}
Note first that the result is true for $S$ of genus $g=0$ or $1$. We prove the proposition by induction on $g$, and 
we may assume $g\ge 2$. 

Let $k\le g-2$ and suppose we have a simplicial map $f\colon S^k\to \OO(S,b_0,b_1)$. 
As $2g+r'-3\ge g-2$, Theorem~\ref{BX} implies that there exists an extension $\hat f$ of $f$
to the disc $D^{k+1}$ with image in $\BB_0(S,\{b_0\},\{b_1\})$, 
and we may assume that $\hat f$ is simplicial by Theorem~\ref{approx}. 
As in the last two proofs, we want to modify $\hat f$ so that its image lies in $\OO(S,b_0,b_1)$. 

For a simplex $\s$ of $\BB_0(S,\{b_0\},\{b_1\})$, with $\s=\lgl a_0,\dots,a_p\rgl$ such that the arcs are ordered $a_0<a_1<\dots<a_p$ in the
anti-clockwise ordering at $b_0$, we can write uniquely
$\s=\s^g*\s^b$ where $\s^g=\lgl a_0,\dots,a_i\rgl$ with $i$ maximal such that the clockwise order at $b_1$ 
starts with $a_0<\dots<a_i$. Thus if $\s=\s^g$, it is a simplex of $\OO(S,b_0,b_1)$. We say that
$\s$ is {\em purely bad} if $\s^g$ is empty. (Note that vertices are always good.)

Let $\s$ be a purely bad $p$--simplex. We claim that the genus of $S\minus\s$ is at least $g-p$. 
If $b_0$ and $b_1$ lie on different boundary components,
this is true regardless of the fact that the simplex is bad:  Cutting along the first arc of $\s$ reduces the number
of boundary components of $S$ without affecting the genus, and subsequent arcs can at most each reduce the genus by one
(by Euler characteristic considerations). 
On the other hand, if $b_0$ and
$b_1$ are in the same boundary component, this is true because $\s$ is bad. 
Indeed, suppose that two arcs $a_i,a_j$ of $\s$ are ordered clockwise both at $b_0$
and at $b_1$. Then their complement $S\minus(a_i\cup a_j)$ is a surface with the same number of boundaries as $S$ and thus
of genus $g-1$ (again because of the Euler characteristic).
The remaining $p-1$ arcs of $\s$ can each reduce the genus by at most one. 

Now we want to remove purely bad simplices from the image of $\hat f$. Let $\s$ be a maximal simplex of $D^{k+1}$ such that $\hat
f(\s)$ is purely bad. As $\s$ is maximal with that property, the link of $\s$ is mapped to $\OO(S\minus\hat f(\s),b_0',b_1')$, where
$b_0'$ and $b_1'$, as shown in Figure~\ref{badorder}, are the copies of $b_0$ and $b_1$ lying between the boundary containing $b_0$ and the first arc of $\s$
at $b_0$ (in the anticlockwise ordering), and between the boundary containing $b_1$ and
the first arc at $b_1$ (in the clockwise ordering).
\begin{figure}[ht]
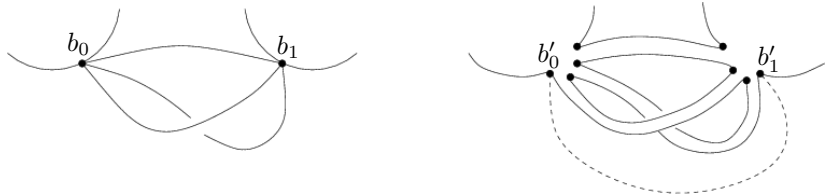
 
\begin{lpic}{badorder(0.45,0.45)}
 \lbl[b]{21,41;$b_0$}
 \lbl[b]{83,40;$b_1$}
 \lbl[b]{160,37;$b_0'$}
 \lbl[b]{225,36;$b_1'$}
\end{lpic}
\caption{Purely bad 2--simplex, and vertex in its link/complement}\label{badorder}
\end{figure}
If $\s$ is a $p$--simplex, $\hat f(\s)$ is a $p'$--simplex with $p'\le p$ and, by the above, 
$S\minus \hat f(\s)$ has genus at
least $g-p'\ge g-p$. The link of $\s$ is a sphere of dimension $k-p$. As $k\le g-2$, we have $k-p\le g-p-2$. 
As $\hat f(\s)$ is bad, $p'\ge 2$ and $S\minus \hat f(\s)$ has genus $\tilde g<g$.  By
induction, the restriction
of $\hat f$ to the link of $\s$ extends to a map $F$ on a  $(k+1-p)$--disc $K$. We replace $\hat f$ 
on $\Star(\s)\simeq \del\s*K$ by $\hat f*F$. 
The new purely bad simplices in the image are faces of $\s$ and hence of smaller dimension. 
The result follows by induction.    
\end{proof}

\section{Closed surfaces}\label{closedsect}

In this section we prove Theorem~\ref{mainclosed}, the stability theorem for surfaces without boundary components. 
The idea, going back to Ivanov \cite{Iva93}, is to build two spectral sequences computing the homology of $\Ga_{g,1}$
and $\Ga_{g,0}$ respectively, so that both sequences have terms involving only mapping class groups of surfaces with
boundaries. Then the map $\delta_g\colon \Ga_{g,1}\to\Ga_{g,0}$ induced by gluing a disc, on the spectral sequences, 
can be identified as the left inverse of the map $\beta$ already computed. We can this way verify that we have  
an isomorphism in a range using the stability theorem for mapping class groups of surfaces with
boundaries. 

Ivanov used the complex of
non-separating curves in a surface. This approach requires: (1) to compute the connectivity of the complex of curves and
(2) handle stabilizers of curves, which are not as well-behaved as stabilizers of arcs. The connectivity can be
deduced from that of the arc complex (see Harer~\cite{Har85}---there is also an unpublished improved argument by
Hatcher-Vogtmann, and an alternative argument by Ivanov \cite{Iva87}). 
Elements in the stabilizer of a circle may rotate the circle, permute circles in a higher simplex and even
flip a circle. This can be
dealt with using appropriate group extensions and comparing associated spectral
sequences (see \cite{Iva93}). 

We will instead describe here Randal-Williams' approach, which is simpler in addition to giving a slightly
better range, though it will require working with the topological group of diffeomorphisms rather than with mapping
class groups, and with semi-simplicial spaces instead of simplicial complexes. 
We start by briefly introducing the language of semi-simplicial spaces.

\medskip

A {\em semi-simplicial space} $X_\bullet$ is a sequence of topological spaces $\{X_p\}_{p\ge 0}$ together with boundary maps 
$d_i\colon X_p\to X_{p-1}$ for each $0\le i\le p$, satisfying the simplicial identity $d_id_j=d_{j-1}d_i$ if $i<j$. 
(So a semi-simplicial space is a
simplicial space without degeneracies.) The space  $X_p$ is the space of $p$--simplices. 

We can define the realization $||X_\bullet||$ of a semi-simplicial space $X_\bullet$ like that of a simplicial set or
simplicial complex by associating a  topological
$p$--simplex $\De^p$ to each  $p$--simplex of $X$: 
$$||X_\bullet||:=\coprod_{p\ge 0}X_p\x \De^p/_\sim$$
with the equivalence given by the face relations $(d_ix,t)\sim (x,d^it)$, where for $t=(t_0,\dots,t_{k-1})\in\De^{k-1}$, with $\sum
t_i=1$, $d^it=(t_0,\dots,t_{i-1},0,t_i,\dots,t_{k-1})$. 

A simplicial space defines a double chain complex with set of $(p,q)$--chains $C_{p,q}(X_\bullet)=C_q(X_p)$, the singular
chains of $X_p$, with vertical differential
$d^V=d_{X_p}$, the differential of $X_p$, and horizontal differential $d^H=\sum_{i=0}^p(-1)^id_i$, the simplicial
differential.  Then $$H_*\big(C_{*,*}(X_\bullet)\,,\,d^H\!\!+\!(-1)^pd^V\big)=H_*(||X_\bullet||).$$ 
We will use below the spectral sequence associated
to the vertical filtration of the double complex, which has $E^1$--term 
$$E^1_{p,q}=H_q(X_p)$$
and converges to the homology of $||X_\bullet||$. 

\smallskip

Let $r\ge 1$ and let  $\del_0S,\dots,\del_{k+r-1} S$ denote the boundary components of the surface $S_{g,k+r}$.  
For $0\le i\le k$, define $d_i\colon \Ga_{g,k+r}\to \Ga_{g,k-1+r}$ to be the map that glues a disc on $\del_iS$. 
These maps make 
$$\xymatrix{B\Ga_{g,\bullet+r}=& \dots\ar@{:}[r] \ar@<1ex>[r] \ar@<-1ex>[r] &  B\Ga_{g,2+r}\ar[r] \ar@<1ex>[r] \ar@<-1ex>[r]& B\Ga_{g,1+r}\ar@<.5ex>[r] \ar@<-.5ex>[r] &  B\Ga_{g,r}}$$
into a semi-simplicial space. (To be precise here, one needs to choose specific compatible identifications of $S_{g,k+r}$ with a disc
glued on its $i$th boundary with $S_{g,k+r-1}$. This is can be done by choosing an appropriate decomposition of the surface into
discs---by linearization, a diffeomorphism from the disc to itself is determined up to homotopy by what
it does on the boundary.)

\begin{thm}\cite{RW09}\label{resolution}
For any $r\ge 0$, we have  $||B\Ga_{g,\bullet+r+1}||\simeq B\Dif(S_{g,r})$.
\end{thm}

A direct consequence of the theorem is  that 
 $||B\Ga_{g,\bullet+r+1}||\simeq B\Ga_{g,r}$ unless $r=0$ and $g=0,1$ as $\Dif(S_{g,r})$ has contractible components in all but these
 two special cases  \cite{EE,ES}. 

\smallskip

We first show how to deduce Theorem~\ref{mainclosed} from this last result.

\begin{proof}[Proof of Theorem~\ref{mainclosed}]
We want to show that the map $$H_*(\delta_g)\colon H_*(\Ga_{g,1})\to H_*(\Ga_{g,0})$$ induced by gluing a disc,  
is an isomorphism for $*\le \frac{2g}{3}$ and a surjection for
$*\le \frac{2g}{3}+1$. For the first two cases $g=0,1$, the non-trivial statement is that $H_*(\delta_g)$ is surjective for $*=1$. 

Let $r=0$ or $1$. 
The spectral sequence for the semi-simplicial space $B\Ga_{g,\bullet+r+1}$ has $E^1$--term
$E^1_{p,q}=H_q(B\Ga_{g,p+r+1})=H_q(\Ga_{g,p+r+1})$ and converges to $H_*(B\Dif(S_{g,r}))$ which is equal to $H_*(\Ga_{g,r})$ 
unless $r=0$ and $g=0,1$. 
Let $\Dif_0(S_{g,0})$ denote the component of the identity in $\Dif(S_{g,0})$. The spectral sequence associated
to the fibration $$B\Dif_0(S_{g,0})\to B\Dif(S_{g,0})\to B\,\pi_0\!\Dif(S_{g,0})=B\Ga_{g,0}$$ 
has $E^2$-term $E^2_{p,q}=H_p(\Ga_{g,0},H_q(B\Dif_0(S_{g,0})))$ converging to $H_{p+q}(B\Dif(S_{g,0}))$. 
As $H_1(B\Dif_0(S_{g,0}))=0$ because $\Dif_0(S_{g,0})$ is
connected, we get $H_i(B\Dif(S_{g,0}))=H_i(\Ga_{g,0})$ for $i=0,1$. 
Hence, for the purpose of proving the theorem, 
we can also use the semi-simplicial space $B\Ga_{g,\bullet+r+1}$ to model the map $\delta_g$ in the cases $g=0,1$.

Gluing a disc on the last boundary of $S_{g,k+2}$ induces a  
simplicial map $B\Ga_{g,\bullet+2}\to B\Ga_{g,\bullet+1}$, which in turn 
induces a map of the corresponding spectral sequences $$E^1_{p,q}=H_q(\Ga_{g,p+2})\ \rar\ \tilde E^1_{p,q}=H_q(\Ga_{g,p+1}).$$ 
As gluing a disc is left inverse to gluing a pair of pants when both the source
and target surfaces have boundaries, this map is surjective on the $E^1$--term, and an isomorphism 
for $q\le \frac{2g}{3}$ by
Theorem~\ref{main}. Hence we get an isomorphism of the targets $H_*(\Ga_{g,1})$ and $H_*(\Ga_{g,0})$ of the spectral sequences 
in all degrees 
$*\le \frac{2g}{3}$ and a surjection in
degree $*\le \frac{2g}{3}+1$: 
\begin{figure}[ht]
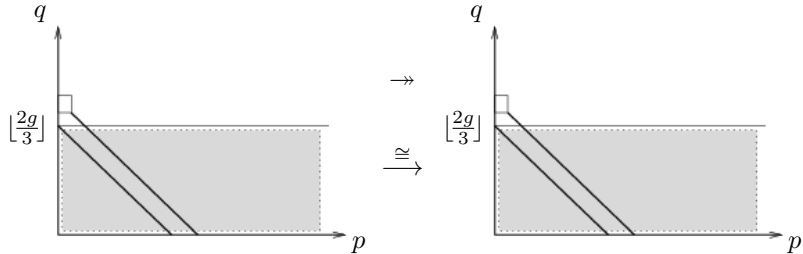

\begin{lpic}{SSiso(0.58,0.58)}
 \lbl[b]{80,15;$\sta{\cong}{\rar}$}
 \lbl[b]{80,35;$\surj$}
 \lbl[b]{70,-3;$p$}
 \lbl[b]{-3,50;$q$}
\lbl[b]{-6,22;$\lfloor\!\frac{2g}{3}\!\rfloor$}
\lbl[b]{170,-3;$p$}
 \lbl[b]{97,50;$q$}
\lbl[b]{94,22;$\lfloor\!\frac{2g}{3}\!\rfloor$}
\end{lpic}
\caption{Map of spectral sequences}\label{SSiso}
\end{figure} 
As shown in Figure~\ref{SSiso}, surjectivity in degree $q=\lfloor\frac{2g}{3}\rfloor+1$ follows from the surjection 
$E^\infty_{0,q}\surj \tilde E^\infty_{0,q}$, which in turns follows from the corresponding surjection in the $E^1$--term by 
the commutativity of the diagram
$$\xymatrix{E^1_{0,q} \ar@{->>}[r] \ar@{->>}[d] & E^\infty_{0,q} \ar[d]\\
\tilde E^1_{0,q} \ar@{->>}[r] & \tilde E^\infty_{0,q}}$$
\end{proof}

\begin{proof}[Proof of Theorem~\ref{resolution}]
Let $\Conf(k,S_{g,r})$ denote the space of configurations of $k$ ordered points in the interior of $S_{g,r}$, each equipped with a framing 
compatible with the orientation of the surface.  
The group $\Dif(S_{g,r})$ acts transitively on $\Conf(k,S_{g,r})$ and the stabilizer of a point is isomorphic to
$\Dif(S_{g,k+r})$. Let $E\Dif(S_{g,r})$ denote a contractible space with a free action of $\Dif(S_{g,r})$. 
By a continuous version of Shapiro's lemma 
$$\Conf(k,S_{g,r})\times_{\Dif(S_{g,r})}E\Dif(S_{g,r})\ \simeq\ B\Dif(S_{g,k+r})\ \simeq\  B\Ga_{g,k+r}$$
where the last equivalence holds by \cite{EE,ES} as long as $k+r>0$. 
We write $\Conf(k,S_{g,r})/\!\!/_{\Dif}:=\Conf(k,S_{g,r})\times_{\Dif(S_{g,r})}E\Dif(S_{g,r})$. In fact, we have an equivalence of 
semi-simplicial spaces $B\Ga_{g,\bullet+r+1}\simeq \Conf(\bullet\!+\!1,S_{g,r})$ for any $r\ge 0$ with  
$$\xymatrix{\Conf(\bullet\!+\!1,S_{g,r})/\!\!/_{\Dif}= \dots\ar[r] \ar@<1ex>[r] \ar@<-1ex>[r]
& \Conf(2,S_{g,r})/\!\!/_{\Dif}\ar@<.5ex>[r] \ar@<-.5ex>[r]& \Conf(1,S_{g,r})/\!\!/_{\Dif}}$$
where, if the framed points at level $k$ are labeled $\overrightarrow{p_0},\dots,\overrightarrow{p_k}$, 
the boundary map $d_i$ forgets the $i$th point
$\overrightarrow{p_i}$.  

To calculate the homotopy type of the above semi-simplicial space, we first consider the semi-simplicial space
$$\xymatrix{\Conf(\bullet\!+\!1,S_{g,r})=&\dots\ar[r] \ar@<1ex>[r] \ar@<-1ex>[r]&\Conf(2,S_{g,r})\ar@<.5ex>[r] \ar@<-.5ex>[r]& \Conf(1,S_{g,r}).}$$
For $r>0$, we can give an explicit retraction $||\Conf(\bullet+1,S_{g,r})||\sta{\sim}{\rar} *$ as follows. 
Gluing a small collar along a boundary
component of $S_{g,r}$ and retracting it shows that each $\Conf(k,S_{g,r})$ is homotopy equivalent to the subspace 
$\Conf_\epsilon(k,S_{g,r})$ of configurations
at least $\epsilon$--distant from that boundary component. Now choose a framed point $\overrightarrow{p}$ in that $\epsilon$--neighborhood. 
We define a retraction from $||\Conf_\epsilon(k,S_{g,r})||\subset ||\Conf(k,S_{g,r})||$ to the 0-simplex represented by
$\overrightarrow{p}$ in $||\Conf(k,S_{g,r})||$: 
A point $x\in ||\Conf_\epsilon(k,S_{g,r})||$ has the form 
$x=[(\overrightarrow{p_0},\dots,\overrightarrow{p_k}),\underline{t}]\in \Conf_\epsilon(k+1,S_{g,r})\x \De^{k}$ for some $k\ge 0$. 
Given $x$, consider the $(k+1)$--simplex  $\{(\overrightarrow{p_0},\dots,\overrightarrow{p_k},\overrightarrow{p})\}\x \De^{k+1}\subset
||\Conf(k,S_{g,r})||$.  In that $(k+1)$--simplex, there is a straight
line from
$[(\overrightarrow{p_0},\dots,\overrightarrow{p_k},\overrightarrow{p}),(\underline{t},0)]
=[(\overrightarrow{p_0},\dots,\overrightarrow{p_k}),\underline{t}]=x$ to
$[(\overrightarrow{p_0},\dots,\overrightarrow{p_k},\overrightarrow{p}),(\underline{0},1)]=[\overrightarrow{p},1]$. Define the
retraction by moving 
at constant speed along that line. This is continuous in $x$. 

As crossing with a contractible space does not change the homotopy type, 
it follows that $ \Conf(\bullet +1 ,S_{g,r})\x E\Dif(S_{g,r})$ is also contractible.
This last semi-simplicial space admits a simplicial diagonal action of
$\Dif(S_{g,r})$ with quotient $ \Conf(\bullet+1,S_{g,r})/\!\!/_{\Dif}$. 
As the action is free, we get that 
$$||\Conf(\bullet+1,S_{g,r})/\!\!/_{\Dif}\,||\ \simeq\ B\Dif(S_{g,r}).$$

\smallskip

The result will follow in the same way for $r=0$ if we can check that  $||\Conf(\bullet+1,S_{g,0})||$ is also contractible.
Gluing a disc on the boundary
component of $S_{g,1}$ induces a simplicial map 
$\Conf(\bullet+1,S_{g,1})\to \Conf(\bullet+1,S_{g,0})$. 
It is a levelwise inclusion and thus a cofibration. 
At each simplicial level $k$, the cofiber $\Conf(k+1,S_{g,1})\big/\Conf(k+1,S_{g,0})$ can be identified with 
$$\bigvee_{i=0}^k\big[S(D^2)\x\Conf(k,S_{g,1})\big]\Big/\big[S(\del D^2)\x\Conf(k,S_{g,1})\big]$$
where the index $i$ in the wedge records which one of the $k+1$ framed points $\overrightarrow{p_0},\dots,\overrightarrow{p_k}$ 
is closest to the center of the glued disc, 
$S(D^2)$ denotes the sphere bundle of $D^2$ and records the position and framing of that point, and 
$\Conf(k,S_{g,1})$ records the position and framing of the $k$ other points.  
A configuration with no point close or closest to the center of the disc in
$S_{g,0}$ is identified with the basepoint. 
When the point closest to the center of the disc moves away from the center, the configuration is identified
with the basepoint, which is why we mod out by $S(\del D^2)\x\Conf(k,S_{g,1})$.

The cofiber at level $k$ can be rewritten as $(S(D^2)/_{S(\del D^2)})\wedge\bigvee_{0}^k\Conf(k,S_{g,1})_+$ and 
the cofiber of the simplicial map is the semi-simplicial space 
$$(S(D^2)/_{S(\del D^2)})\wedge\bigvee_{i=0}^\bullet\Conf(\bullet,S_{g,1})_+$$ with boundary maps 
on $\bigvee_{0}^\bullet\Conf(\bullet,S_{g,1})_+$ induced from  
$\Conf(\bullet+1,S_{g,0})$: Denoting the points in a configuration in the $j$th summand 
$\overrightarrow{p_0},\dots,\overrightarrow{p_{j-1}},\ $ $\overrightarrow{p_{j+1}},\dots,\overrightarrow{p_k}$, the boundary map $d_i$ on this summand forgets 
the $i$th point $\overrightarrow{p_i}$, unless $j=i$ in which case it just maps to the basepoint. 

By the same argument as above, 
the semi-simplicial space $\bigvee_{0}^\bullet\Conf(\bullet,S_{g,1})_+$ 
is equivalent to $\bigvee_{0}^\bullet\Conf_\epsilon(\bullet,S_{g,1})_+$
 and we can again define a retraction of this subspace by
choosing a point $\overrightarrow{p}$ in the $\epsilon$--neighborhood of the boundary. This time, a point in the realization has the
form  $x=[j,(\overrightarrow{p_1},\dots,\overrightarrow{p_{j-1}},\overrightarrow{p_{j+1}},\dots,\overrightarrow{p_k}),\underline{t}]$ and we use the straight line in the
simplex
$\{(j,(\overrightarrow{p_1},\dots,\overrightarrow{p_{j-1}},\overrightarrow{p_{j+1}},\dots,\overrightarrow{p_k},\overrightarrow{p}))\} \x \De^k$, from $x$ to the point
$[j,\overrightarrow{p},1]$, which is identified with the basepoint for all $j$.

Hence the cofiber of the map  $\Conf(\bullet+1,S_{g,1})\to \Conf(\bullet+1,S_{g,0})$ is contractible. As
the collapsed space $\Conf(\bullet+1,S_{g,1})$ was contractible, we have that $\Conf(\bullet+1,S_{g,0})$ is also 
contractible and the result follows for $r=0$. 
\end{proof}

\section{Appendix: Simplicial complexes}

This short section gives the background material on simplicial complexes and piecewise linear topology needed in the
rest of the paper. In particular, we consider joins of complexes and state the simplicial approximation theorem.

\medskip

Combinatorially, a {\em simplicial complex} $X=(X_0,\F)$ is a set of vertices $X_0$ together with a collection $\F$ 
of subsets of $X_0$ closed under taking subsets and
containing all the singletons. The
subsets of cardinality $p+1$ are called the {\em $p$--simplices} of $X$. 
If $\s=\lgl x_0,\dots,x_p\rgl$ is a $p$--simplex of $X$, the subsets $\lgl x_{i_0},\dots,x_{i_k}\rgl$ of $\s$ are called its {\em
  faces}.  

To a simplicial complex $X$, one can associate a topological space, its {\em realization}, denoted $|X|$ or just $X$ again, build as
follows: $|X|$ has a 0-cell for each vertex of $X$, a 1-cell between any to vertices $v,w$ such that $\lgl v,w\rgl$ is a simplex of
$X$, and more generally a $p$-simplex $\Delta^p$ for each $p$-simplex $\lgl v_0,\dots,v_p\rgl$ of $X$ with its codimension one faces 
identified with the simplices associated to the faces $\lgl v_0,\dots,\widehat{v_j},\dots,v_p\rgl$ of the simplex. 
When we talk about topological properties of a simplicial complex $X$, such as its connectivity, we mean the corresponding
property for this associated topological space. 

\medskip

The {\em join} $X*Y$ of two simplicial complexes $X$ and $Y$ is the
simplicial complex with vertices $X_0\sqcup Y_0$ and a $(p+q+1)$--simplex 
$\s_X*\s_Y=\lgl x_0,\dots,x_p,y_0,\dots,y_q\rgl$ for each 
$p$--simplex $\s_X=\lgl x_0,\dots,x_p\rgl$  of $X$ and $q$--simplex $\s_Y=\lgl y_0,\dots,y_q\rgl$ of
$Y$. Note that $|X*Y|=|X|*|Y|$, i.e. the realization of the join complex is the (topological) join of the realization of the two
complexes. This follows from the fact that it is true for each pair of simplices.

\medskip

Recall that a space (or simplicial complex) $X$ is called {\em $n$--connected} if $\pi_i(X)=0$ for all $i\le n$ (where
$\pi_i(X):=\pi_i(|X|)$ if $X$ is a simplicial complex). For $n=-1$, we use the convention that {\em $(-1)$--connected} means
non-empty. (For $n\le -2$, $n$--connected is a void property.) 
Note that, by Hurewicz theorem, a simply connected space $X$ is $n$-connected, $n\ge 2$, if and only if $H_*(X)=0$ for $0<*\le n$. 

The following proposition, which goes back at least to Milnor, tells us how to compute the connectivity of a join in
terms of the connectivity of the pieces.

\begin{prop}\cite[Lem 2.3]{Mil}\label{joinconn}
Consider the join $X=X_1*\dots*X_k$ of $k$ non-empty spaces. If each $X_i$ is $n_i$--connected, then $X$ is 
$\big(\big(\sum_{i=1}^k(n_i+2)\big)-2\big)$--connected. 
\end{prop}

Note that the lemma implies that $X$ is contractible whenever some $X_i$ is contractible.

\medskip

Given a simplex $\s$ of a simplicial complex $X$, the (closed) {\em star of $\s$}, $\Star(\s)$, 
is the subcomplex of $X$ of simplices containing $\s$, together with their faces. 
The {\em link of $\s$}, $\link(\s)\subset \Star(\s)$, is the subcomplex of the star of simplices disjoint from $\s$. 
The link can also be described as the subcomplex of simplices $\tau$ disjoint from $\s$ such that $\tau*\s$ is again a simplex of
$X$, and 
 $$\Star(\s)=\link(\s)*\s$$
where $\s$ in the formula denotes the subcomplex of $X$ defined by $\s$ and its faces. 

\medskip

A {\em PL (or manifold) triangulation} $K$ of an $n$--manifold $M$ 
is a simplicial complex $K$ such that $|K|\cong M$ and with the
property that the link of a $p$--simplex $\s$ of $K$ is PL homeomorphic to the boundary of an $(n-p)$--simplex 
if $\s$ not included in $\del K$, 
 or to an $(n-p-1)$--simplex if $\s\subset \del K$.

\begin{lem}\label{links}(See for example \cite[Lem 1.13]{Hu}\label{D*S}.)
If $D^n$ and $S^n$ denote PL triangulations of the $n$--disc and $n$--sphere, then 
\begin{enumerate}[$(i)$]
\item $|D^n*D^m|$ is an $(n+m+1)$--disc,
\item $|D^n*S^m|$ is an $(n+m+1)$--disc, and 
\item $|S^n*S^m|$ is an $(n+m+1)$--sphere. 
\end{enumerate}
\end{lem}

Applying the Lemma to a PL triangulation $K$ of an $n$--manifold, we get that, for a $p$--simplex $\s$, as  
$\Star(\s)=\link(\s)*\s$, it is of the type $S^{n-p-1}*D^p$ (or $D^{n-p-1}*D^p$ if $\s\subset\del K$), and hence the star  
of any simplex is an $n$--disc. Note moreover that, if $\s\not\subset \del K$, the boundary of $\Star(\s)$ is the $(n-1)$--sphere $\del\s*\link(\s)$.

\medskip

A main theorem we need for the connectivity results in Section~\ref{connectivity} is the following: 

\begin{thm}[Simplicial approximation]\cite{Zee}\label{approx} 
Let $K,L$ be finite simplicial complexes, and $L$ a subcomplex of $K$. Let $f\colon |K|\to|X|$ be a continuous map such that the restriction
$f|_L$ is a simplicial map from $L$ to $X$. Then there exists a relative subdivision $(K_r,L)$ of $(K,L)$ and a simplicial 
map $g\colon K_r\to X$ such that $g|_L=f|_L$ and $g$ is homotopic to $f$ keeping $L$ fixed.    
\end{thm}

We use this theorem in Section~\ref{connectivity} to approximate any map from a sphere into a simplicial complex by a simplicial map,
and to approximate a null-homotopy of such a map, now simplicial, by a simplicial map from the disc with a triangulation extending
that of the sphere. Hence we apply the theorem for the cases $(K,L)=(S^k,\emp)$ and $(K,L)=(D^{k+1},S^k)$. Note that the complexes $X$ we
work with are usually not finite, but when applying the theorem, we can restrict to the (finite) subcomplex of $X$ containing the image
of the sphere or the disc. We also need the triangulations of the spheres and discs to be PL triangulations, and this can be obtained
by choosing some PL triangulation of $S^k$ (resp. $D^{k+1}$) and applying the theorem to it, noting that the subdivision (resp. relative
subdivision) preserves the PL property.

\end{document}